\documentclass[10pt,a4paper]{article}
\usepackage[british]{babel}
\usepackage{amssymb,amsmath,amsthm}
\usepackage{cite}
\usepackage{graphicx}
\usepackage[bookmarks=true, bookmarksopen=true, bookmarksopenlevel=4]{hyperref}
\usepackage{color}
\usepackage{pgf,tikz}
\usetikzlibrary{decorations.pathreplacing,matrix,calc,arrows,shapes.geometric}

\hypersetup{pdftitle={A priori Analysis of Space-Time Trefftz
    Discontinuous Galerkin Methods for Wave Problems},
  pdfauthor={F. Kretzschmar, A. Moiola, I. Perugia, S. M. Schnepp},
  colorlinks=true, linkcolor=myblue,  citecolor=myblue, filecolor=myblue,   urlcolor=myblue,}

\allowdisplaybreaks[4]

\allowdisplaybreaks[4]
\definecolor{myblue}{rgb}{0,0,0.6}         
\definecolor{gray}{rgb}{0.5,0.5,0.5}

\newcommand*{\der}[2]{\frac{\partial #1}{\partial #2}}
\newcommand{\di}{\,\mathrm{d}}                              
\newcommand{\iin}{\;\text{in}\;}
\newcommand{\oon}{\;\text{on}\;}
\newcommand{\deO}{{\partial\Omega}}

\newcommand*{\N}[1]{\left\|#1\right\|}
\newcommand*{\abs}[1]{\left|#1\right|}
\newcommand{\Tnorm}[1]{|||#1|||}
\newcommand*{\jmp}[1]{[\![#1]\!]}                     
\newcommand*{\mvl}[1]{\{\!\!\{#1\}\!\!\}}             

\newcommand{\dive} {\mathop{\rm div}\nolimits}

\DeclareMathOperator{\im}{Im} %

\DeclareMathOperator{\spn}{span}

\newcommand{\Uu}[1]{{\mathbf{#1}}}                   

\newcommand{\IC}{\mathbb{C}}

\newcommand{\IN}{\mathbb{N}}\newcommand{\IP}{\mathbb{P}}
\newcommand{\IR}{\mathbb{R}}

\newcommand{\bg}{{\Uu g}}

\newcommand{\bn}{{\Uu n}}

\newcommand{\bu}{{\Uu u}}
\newcommand{\bx}{{\Uu x}}

\newcommand{\bE}{{\Uu E}}
\newcommand{\bH}{{\Uu H}}

\newcommand{\bN}{{\Uu N}}

\newcommand{\bT}{{\Uu T}}
\newcommand{\bV}{{\Uu V}}

\newcommand{\bsigma}{{\boldsymbol \sigma}}\newcommand{\btau}{{\boldsymbol\tau}}

           \newcommand{\bzero}{\Uu{0}}

\newcommand{\calE}{{\mathcal E}}\newcommand{\calF}{{\mathcal F}}

\newcommand{\calO}{{\mathcal O}}
\newcommand{\calR}{{\mathcal R}}
\newcommand{\calS}{{\mathcal S}}\newcommand{\calT}{{\mathcal T}}

\newcommand{\malpha}{{\boldsymbol{\alpha}}}

  \theoremstyle{plain}
\newtheorem{theorem}{Theorem}[section]
\newtheorem{lemma}[theorem]{Lemma}
\newtheorem{prop}[theorem]{Proposition}
\newtheorem{proposition}[theorem]{Proposition}
\newtheorem{cor}[theorem]{Corollary}
  \theoremstyle{definition}
  \theoremstyle{remark}
\newtheorem{rem}[theorem]{Remark}
\newcommand{\beq}{\begin{equation}}      \newcommand{\eeq}{\end{equation}}
\newcommand{\beqs}{\begin{equation*}}    \newcommand{\eeqs}{\end{equation*}}
\newcommand{\bit}{\begin{itemize}}       \newcommand{\eit}{\end{itemize}}
\newcommand{\ben}{\begin{enumerate}}     \newcommand{\een}{\end{enumerate}}
\newcommand{\bal}{\begin{align}}         \newcommand{\eal}{\end{align}}
\newcommand{\bals}{\begin{align*}}       \newcommand{\eals}{\end{align*}}
\newcommand{\bse}{\begin{subequations}}	 \newcommand{\ese}{\end{subequations}}
\newcommand{\bpr}{\begin{proposition}}   \newcommand{\epr}{\end{proposition}}
\newcommand{\bre}{\begin{remark}}        \newcommand{\ere}{\end{remark}}
\newcommand{\bpf}{\begin{proof}}         \newcommand{\epf}{\end{proof}}
\newcommand{\ble}{\begin{lemma}}         \newcommand{\ele}{\end{lemma}}
\newcommand{\bco}{\begin{corollary}}     \newcommand{\eco}{\end{corollary}}
\newcommand{\bex}{\begin{example}}       \newcommand{\eex}{\end{example}}
\newcommand{\bth}{\begin{theorem}}       \newcommand{\enth}{\end{theorem}}

\newcommand{\Fh}{\calF_h}
\newcommand{\Fhor}{{\Fh^{\text{hor}}}}
\newcommand{\Fver}{{\Fh^{\text{ver}}}}
\newcommand{\FT}{{\Fh^T}}
\newcommand{\FO}{{\Fh^0}}
\newcommand{\FL}{{\Fh^L}}
\newcommand{\FR}{{\Fh^R}}

\newcommand\Ehp{E_{hp}}
\newcommand\Hhp{H_{hp}}
\newcommand\hEhp{\widehat E_{hp}}
\newcommand\hHhp{\widehat H_{hp}}
\newcommand{\vE}{v_E}
\newcommand{\vH}{v_H}
\newcommand{\wE}{w_E}
\newcommand{\wH}{w_H}
\newcommand{\Th}{{(\calT_h)}}

\newcommand{\tvE}{{{\widetilde v}_E}}
\newcommand{\tvH}{{{\widetilde v}_H}}
\newcommand{\tphi}{{\widetilde\phi}}
\newcommand{\tpsi}{{\widetilde\psi}}
\newcommand{\ODp}{{\Omega_D^+}}
\newcommand{\ODm}{{\Omega_D^-}}
\newcommand{\ODpm}{{\Omega_D^\pm}}

\newcommand{\epsil}{{\varepsilon}}

\newcommand{\Cstab}{C_{\mathrm{stab}}}
\newcommand{\derp}[2]{{\frac{\partial #1}{\partial #2}}}

\newcommand\bEhp{\mathbf{E}_{hp}}
\newcommand\bHhp{\mathbf{H}_{hp}}
\newcommand\hbEhp{\widehat \bE_{hp}}
\newcommand\hbHhp{\widehat \bH_{hp}}
\newcommand\bvE{\mathbf{v}_E}
\newcommand\bvH{\mathbf{v}_H}
\newcommand{\Fd}{{\Fh^\partial}}
\newcommand\Vhp{v_{hp}}
\newcommand\Shp{\bsigma_{hp}}

\title{A Priori Error Analysis of Space--Time Trefftz Discontinuous Galerkin Methods for Wave Problems}

\author{Fritz Kretzschmar\thanks{Graduate School of Computational Engineering, TU Darmstadt, Dolivostrasse 15, 64293 Darmstadt, Germany}\;\thanks{Institut f\"ur Theorie Elektromagnetischer Felder, TU Darmstadt, Schlossgartenstrasse 8, 64289 Darmstadt, Germany}, 
Andrea Moiola\thanks{Department of Mathematics and Statistics, University of Reading, Whiteknights PO Box 220, Reading RG6 6AX, UK (\texttt{a.moiola@reading.ac.uk}) (corresponding author)},
Ilaria Perugia\thanks{Faculty of Mathematics, University of Vienna,
  1090 Vienna, Austria, and Department of Mathematics, University of
  Pavia, 27100 Pavia, Italy (\texttt{ilaria.perugia@univie.ac.at})}, 
Sascha M.~Schnepp\footnote{Institute of Geophysics, Department of Earth Sciences, ETH Zurich, CH-8092 Zurich, Switzerland}
}

\date{\today}

\begin{document}

\maketitle

\begin{abstract}
We present and analyse a space--time discontinuous Galerkin method for wave propagation problems.
The special feature of the scheme is that it is a Trefftz method, namely that trial and test functions are solution of the partial differential equation to be discretised in each element of the (space--time) mesh.
The method considered is a modification of the discontinuous Galerkin schemes of Kretzschmar et al.\ \cite{KSTW2014} and of Monk and Richter \cite{MoRi05}.
For Maxwell's equations in one space dimension, we prove stability of the method, quasi-optimality, best approximation estimates for polynomial Trefftz spaces and (fully explicit) error bounds with high order in the meshwidth and in the polynomial degree. 
The analysis framework also applies to scalar wave problems and Maxwell's equations in higher space dimensions.
Some numerical experiments demonstrate the theoretical results proved and the faster convergence compared to the non-Trefftz version of the scheme.

\


\

\end{abstract}
 
\section{Introduction}

In this paper we analyse the Trefftz discontinuous Galerkin (Trefftz-DG) finite
element method for the space--time discretisation of time-dependent
wave propagation problems. 

Space--time finite element approximations of time-dependent wave
problems constitute a 
viable alternative to methods 
based on finite element discretisations in space combined with time-stepping schemes. 
They provide a setting where high-order
approximation in both space and time is obtained by simply
increasing the order of the 
basis functions, spectral convergence of the error in the space--time domain can be achieved by $p$-refinement, and
the $hp$-refinement techniques developed for spatial discretisations
can be directly imported to space--time meshes and polynomial spaces.
In particular, when local mesh refinement in space--time is
implemented, the smallest space elements do not put constraints on the
CFL condition in the rest of the domain.
Space--time finite elements with continuous space discretisations were introduced in \cite{HughesHulbert1998,HulbertHughes1990,French1993,FrenchPeterson1996,ThompsonPinsky1996,Jo93} (several references to earlier works can be found in \cite{HulbertHughes1990}); space--time
DG methods have been studied in \cite{FalkRichter1999,MoRi05,VVV2002,CostanzoHuang05,Lilienthal:2014fs,GrMo14}.

In order to reduce the number of degrees of freedom needed to obtain a given accuracy, 
the space--time DG approach can be combined with the use of Trefftz approximating spaces, 
namely discrete spaces whose elements are piecewise solutions of the equation to be discretised.

While the literature on Trefftz finite elements for time-harmonic wave
propagation problems is nowadays quite developed (see e.g.\ \cite{CED98,MOW99,TEF06,FHH03,GAB07,HMK02,Deckers2014} 
for different approaches using Trefftz-type basis functions and
e.g.\ \cite{BUM07,GHP09,PVersion,HMP12} for theoretical analyses),
Trefftz methods for time-dependent wave problems are relatively new.
For time-dependent scalar wave problems, 
a global Trefftz approach was first introduced in \cite{Maciag05}.
Space--time Trefftz methods enforcing the continuity
constraints by Lagrange multipliers have been designed, based on the second order in time
formulation \cite{PFT09,WTF14}, while a Trefftz-DG method has been introduced in \cite{KSTW2014} for the electromagnetic wave problem in the one-dimensional spatial case and extended in~\cite{EKSTWtransparent}
to the multidimensional case and to transparent boundary conditions. 
Numerical tests have shown that
these methods actually deliver high-order space--time convergence, but
no theoretical analysis was available. 

Here, we focus on the discontinuous Galerkin approach, and we aim at
developing an error analysis of space--time Trefftz-DG methods, based
upon the framework developed in \cite{PVersion} for time-harmonic wave
problems.

Following \cite{KSTW2014}, the model problems we consider are the
time-dependent Maxwell equations in one space dimension, 
with piecewise constant material coefficients and with
either Dirichlet or Robin boundary conditions; they also cover the
case of 1D scalar wave problems formulated as a first order system
(see Remark~\ref{rem:WaveEq}).
We prove that the Trefftz-DG method is well-posed and quasi-optimal,
  as its bilinear form is coercive and continuous in a DG norm, that
  the $L^2$ norm of the error is controlled, and that the scheme is
  dissipative.
Since the Trefftz-DG method we consider is not of interior
  penalty type, no inverse estimates are needed, thus our scheme and
  its analysis also work with non-polynomial Trefftz approximations.
The analysis framework of section~\ref{s:Analysis} immediately extends 
to both scalar wave problems and Maxwell's equations in higher space
dimensions, see Remark~\ref{rem:HigherD}. 
What is specific to the 1D spatial case, instead, are the best
approximation properties of space--time Trefftz finite element spaces,
and therefore 
the error convergence rates
(see sections~\ref{s:BestApprox} and~\ref{s:ConvRates}).

For the case of polynomial approximations, we prove that the Trefftz-DG method has a much better asymptotic behaviour in terms of accuracy per
number of degrees of freedom, for both the $h$- and the $p$-versions, than standard methods 
using full polynomial discrete spaces.
The approximation bounds for Trefftz spaces and those for full polynomial
spaces
have the same order
$h^{p+1}$, with $h$ the meshsize and $p$ the polynomial degree, but a
local full polynomial space of degree $p$ for the approximation of
$(E,H)$  has dimension $(p+1)(p+2)=\calO(p^2)$,
while the corresponding Trefftz space has only dimension
$2p+2=\calO(p)$.
The constants in the final error bounds for the Trefftz-DG method (see Theorem \ref{th:Convergence}) are completely explicit.
{A further advantage of the use of a Trefftz method is that its variational formulation only involves integrals on the skeleton of the space--time mesh, thus avoiding the computational burden of the calculation of volume integrals.}
We also present some numerical tests confirming these theoretical results, highlighting
the faster convergence (both in the meshwidth and in the polynomial degree) compared to the non-Trefftz version of the method, and 
the mild dependence of the numerical error on the flux stabilisation parameters,
which we introduce for analysis purposes.

The article is structured as follows.
In section \ref{s:IVP} we introduce the initial boundary value problems to be discretised
and in section \ref{s:Discretisation} we describe the DG formulation and its Trefftz version.
In section \ref{s:Analysis} we prove the well-posedness of the Trefftz-DG method and the a priori error bounds in DG and $L^2$ norms.
In section \ref{s:BestApprox} we define local polynomial Trefftz
spaces and prove best approximation estimates in anisotropic space--time Sobolev norms
and in section~\ref{s:ConvRates} we combine the previous results into $hp$-error bounds for the Trefftz-DG scheme.
Finally, in section~\ref{s:numer-exper} we show the results of some
numerical experiments and in section \ref{s:conclusions} we outline some possible extensions of the scheme and of its analysis.
Appendix \ref{appendix} contains a different proof of the stability bound needed to control the $L^2$ norm of the error; this proof gives a better constant than that in section \ref{ss:L2} but holds only for constant coefficients.

\section{Model problems}\label{s:IVP}

In this section, we introduce the model problems we are going to
  consider in the rest of this paper, namely the time-dependent
  Maxwell equations in one space dimension, with either
  perfectly electrically conducting or absorbing boundary conditions.
We refer to Remark~\ref{rem:HigherD} below for the
  three-dimensional case, while the case of the acoustic wave problem is
  discussed in Remark~\ref{rem:WaveEq} (in one space dimension) and again
  in Remark~\ref{rem:HigherD} (in any space dimension).

Given a space domain $\Omega=(x_L,x_R)$ and a time domain $I=(0,T)$,
we set $Q:=\Omega\times I$. We denote by $\bn_Q=(n_Q^x,n_Q^t)$ the
outward pointing unit normal vector on $\partial Q$.

We assume the electric permittivity $\epsil=\epsil(x)$ and the magnetic
permeability $\mu=\mu(x)$ to satisfy $\epsil(x)\ge\epsil_*>0$,
$\mu(x)\ge\mu_*>0$, and to be piecewise constant in $\Omega$.
The speed of light in the material is
$c(x):=(\epsil(x)\mu(x))^{-1/2}$.

We consider the Maxwell equations for an electromagnetic wave
propagating along the direction $x$, in terms of the one-component electric and magnetic
fields $E=E_y$ and $H=H_z$, respectively.

In the case of Dirichlet boundary conditions, the corresponding initial boundary value problem reads as
\begin{align}\label{eq:IVP}
\begin{aligned}
&\der{E}x+\der{(\mu H)}t = 0 &&\iin Q,\\
&\der{H}x+\der{(\epsil E)}t = J &&\iin Q,\\
&E(\cdot,0)=E_0, \quad
H(\cdot,0)=H_0 &&\oon \Omega,\\
&E(x_L,\cdot)=E_L,\quad
E(x_R,\cdot)=E_R &&\oon I,
\end{aligned}
\end{align}
assuming sufficiently regular data $J,E_0,H_0,E_L,E_R$.
The case $E_L=E_R=0$ models a perfectly electric conductor.

We also consider the case of absorbing (Robin) boundary conditions:
%
\begin{align}\label{eq:RobinIVP}
\begin{aligned}
&\der{E}x+\der{(\mu H)}t = 0 &&\iin Q,\\
&\der{H}x+\der{(\epsil E)}t = J &&\iin Q,\\
&E(\cdot,0)=E_0, \quad
H(\cdot,0)=H_0 &&\oon \Omega,\\
&\epsil^{1/2} E(x_L,\cdot)+\mu^{1/2}H(x_L,\cdot)=g_L&&\oon I,\\
&\epsil^{1/2} E(x_R,\cdot)-\mu^{1/2}H(x_R,\cdot)=g_R &&\oon I.
\end{aligned}
\end{align}

\begin{rem}\label{rem:WaveEq}
The case of the scalar wave equation in one space dimension
\begin{equation}\label{eq:WaveEq}
\der{{}^2U}{t^2} -c^2 \der{{}^2U}{x^2}  =c^2 J
\end{equation}
can be traced back to either~\eqref{eq:IVP} or~\eqref{eq:RobinIVP} by
setting
$$E=\der{U}t,\quad H=-\der{U}x\quad \text{and} \quad \mu=1,$$
(where $c^2=1/\epsil\mu$ as above).
The Dirichlet boundary conditions become $\der U t(x_L,\cdot)=E_L$ and
$\der U t(x_R,\cdot)=E_R$.
The absorbing boundary conditions become
$\epsil^{1/2}\der U t+\mu^{1/2}\der U x = g_L$,
$\epsil^{1/2}\der U t-\mu^{1/2}\der U x = g_R$.
\end{rem}

In the following, we introduce the space--time Trefftz-DG method and develop its analysis in the case of the initial boundary value problem with Dirichlet boundary conditions~\eqref{eq:IVP}. 
We address the case of the problem with Robin boundary conditions~\eqref{eq:RobinIVP} in section~\ref{ss:Robin} and in Remarks~\ref{rem:RobinConvRates} and~\ref{rem:RobinAnalytic}.

\section{Space--time DG discretisation}\label{s:Discretisation}

Let $K\subset Q$ be a Lipschitz subdomain such that $\epsil$ and
$\mu$ are constant in $K$; we denote by $\bn_K=(n_K^x,
n_K^t)$ the outward pointing unit normal vector on $\partial K$.
Assume $E,H\in H^1(K)$.
Multiplying the first and the second equation of
\eqref{eq:IVP} by the test functions $\vH ,\vE \in H^1(K)$ respectively, and integrating by parts we obtain
\begin{align*}
-\iint_K &\Big(E\der{\vH }x+\mu H\der{\vH }t
+H \der{\vE }{x}+\epsil E\der{\vE }t\Big)\di x\di t
\nonumber\\
&+\int_{\partial K} \Big(
(E\vH +H\vE )n^x_K+(\mu H\vH +\epsil E\vE )n^t_K\Big) \di s
=\iint_KJ \vH \di x\di t.
\end{align*}
(Here and in the following we omit writing the trace operator
$\mathrm
{Tr}: H^1(K)\to L^2(\partial K)$.)

\subsection{Mesh and DG notation}\label{ss:Mesh}
We introduce a mesh $\calT_h$ on $Q$, such that its elements are rectangles
with sides parallel to the space and time axes, and all the discontinuities of the parameters $\epsil$
and $\mu$ lie on interelement boundaries (note that the method described in this paper can be generalised to allow discontinuities lying inside the elements as in \cite{KSTW2014}).
The mesh may have hanging nodes.

We denote with $\Fh=\bigcup_{K\in\calT_h}\partial K$ the mesh
skeleton and its subsets:
\begin{align*}
\Fhor&:=\text{the union of the internal horizontal element sides ($t=$ constant)},
\\
\Fver&:=\text{the union of the internal vertical element sides ($x=$ constant)},
\\
\FO&:=
[x_L,x_R]\times\{0\},\\
\FT&:=
[x_L,x_R]\times\{T\},\\
\FL&:=
\{x_L\}\times[0,T],\\
\FR&:=
\{x_R\}\times[0,T].
\end{align*}

We define the following broken Sobolev space:
$$H^1(\calT_h):=\big\{v\in L^2(Q),\; v|_K\in H^1(K)\; \forall K\in \calT_h\big\}$$
and the averages and jumps of functions $\phi\in H^1(\calT_h)$:
\begin{align*}
\mvl{\phi}&:=\frac12\big(\phi_{|_{K_1}}+\phi_{|_{K_2}}\big) 
&& \oon (\partial K_1\cap \partial K_2)\subset \Fver,\\
\jmp{\phi}_x&:= \phi_{|_{K_1}}n_{K_1}^x+\phi_{|_{K_2}} n_{K_2}^x
&& \oon (\partial K_1\cap \partial K_2)\subset \Fver,\\
\jmp{\phi}_t&:= \phi_{|_{K_1}}n_{K_1}^t+\phi_{|_{K_2}} n_{K_2}^t
&& \oon (\partial K_1\cap \partial K_2)\subset \Fhor.
\end{align*}
We note that $\jmp{\phi}_x$ is the trace of $\phi$ from
the left minus that from the right; $\jmp{\phi}_t$ is the trace from
lower times minus that from higher times.
On $\Fhor$, we denote by $\phi^-$ the trace of $\phi\in
H^1(\calT_h)$ taken from the adjacent element with lower time.

\subsection{DG formulation}\label{ss:DG}

We introduce a (vector) finite dimensional subspace $\bV_p(\calT_h)\subset H^1(\calT_h)^2$.
The space--time DG discretisation of the initial-boundary value problem
\eqref{eq:IVP} consists in finding $(\Ehp,\Hhp)\in \bV_p(\calT_h)$ such
that, for all $K\in \calT_h$ 
and for all $(\vE ,\vH )\in \bV_p(\calT_h)$, it holds
\begin{align}
-\iint_K &\Big(\Ehp\der{\vH }x+\mu \Hhp\der{\vH }t
+\Hhp \der{\vE }{x}+\epsil \Ehp\der{\vE }t\Big)\di x\di t
\nonumber\\
&+\int_{\partial K} \Big(
(\hEhp \vH +\hHhp \vE )n^x_K+(\epsil \hEhp \vE +\mu \hHhp \vH )n^t_K\Big) \di s
=\iint_KJ \vH \di x\di t.
\label{eq:DiscrPbl}
\end{align}
The numerical fluxes $\hEhp$ and $\hHhp$ are defined on the mesh
skeleton $\Fh$ as
\begin{align}\label{eq:Fluxes}
\hEhp
:=\begin{cases}
\Ehp^- & \oon \Fhor,\\
\Ehp &\oon\FT,\\
E_0 &\oon \FO,\\
\mvl{\Ehp}+\beta \jmp{\Hhp}_x &\oon \Fver,\\
E_L &\oon \FL,\\
E_R &\oon \FR,
\end{cases}
\qquad 
\hHhp
:=\begin{cases}
\Hhp^- & \oon \Fhor,\\
\Hhp &\oon \FT,\\
H_0 &\oon \FO,\\
\mvl{\Hhp}+\alpha \jmp{\Ehp}_x &\oon \Fver,\\
\Hhp-\alpha(\Ehp-E_L) &\oon \FL,\\
\Hhp+\alpha(\Ehp-E_R) &\oon \FR,
\end{cases}
\end{align}
where  $\alpha\in L^\infty(\Fver\cup\FL\cup\FR)$ and $\beta\in L^\infty(\Fver)$ are two
positive flux parameters.
The fluxes are chosen as upwind fluxes on horizontal mesh edges,
and centred fluxes with the addition of a penalty jump stabilisation on vertical
edges; other choices are possible.

Summing \eqref{eq:DiscrPbl} over all the elements $K\in\calT_h$, we
obtain the variational formulation
\begin{align}
&\text{seek}\;(\Ehp,\Hhp)\in \bV_p(\calT_h)\; \text{such that}\nonumber\\
&a_{DG}(\Ehp,\Hhp;\vE ,\vH )=\ell_{DG}(\vE ,\vH )\qquad \forall (\vE ,\vH )\in \bV_p(\calT_h),
\label{eq:VF}
\end{align}
where
\begin{align}\label{eq:Bilin}
&\begin{aligned}
a_{DG}(\Ehp,\Hhp;&\vE ,\vH ):=-\sum_{K\in\calT_h}\iint_K
\Big(\Ehp\der{\vH }x+\mu \Hhp\der{\vH }t
+\Hhp \der{\vE }{x}+\epsil \Ehp\der{\vE }t\Big)\di x\di t\\
&+\int_{\Fhor}(\epsil\Ehp^-\jmp{\vE }_t+\mu\Hhp^-\jmp{\vH }_t)\di x
+\int_\FT (\epsil\Ehp \vE +\mu\Hhp \vH )\di x\\
&+\int_\Fver \big( \mvl{\Ehp}\jmp{\vH }_x+\mvl{\Hhp}\jmp{\vE }_x
+\alpha\jmp{\Ehp}_x\jmp{\vE }_x+ \beta\jmp{\Hhp}_x\jmp{\vH }_x\big)\di t\\
&+\int_\FL \big(-\Hhp \vE +\alpha \Ehp \vE \big) \di t
+\int_\FR \big(\Hhp \vE +\alpha \Ehp \vE \big) \di t,
\end{aligned}
\\
&\begin{aligned}
\ell_{DG}(\vE ,\vH ):=&
\iint_Q J \vH \di x\di t
+\int_\FO (\epsil E_0\vE  +\mu H_0 \vH )\di x\\
&+\int_\FL E_L(\vH +\alpha \vE )\di t
+\int_\FR E_R(-\vH +\alpha \vE )\di t.
\end{aligned}\nonumber
\end{align}

\begin{rem}\label{rem:MonkRichter}
In \cite{MoRi05}, a space--time DG discretisation of linear
hyperbolic systems in $N$ dimensions is introduced. 
The formulation \eqref{eq:VF} with $\alpha=\beta=\frac12$ and $\bV_p(\calT_h)=\IP^p(\calT_h)^2$, the space of piecewise polynomials of degree at most $p$ on $\calT_h$,  is a particular case of that of \cite{MoRi05}.
More precisely, in the notation of \cite{MoRi05},
assume the initial boundary value problem to be posed in one space dimension ($N=1$),
and to have a linear hyperbolic system of dimension $m=2$ 
with time derivative coefficient matrix 
$A=(\begin{smallmatrix}1 &0\\0&1\end{smallmatrix})$,
and space derivative coefficient matrix
$A_1=(\begin{smallmatrix}0&1\\1&0\end{smallmatrix})$.
The solution will be the vector field $\bu=(E,H)$.
The matrix defining the Dirichlet boundary conditions takes
values
$\mathsf N|_{\FL}=(\begin{smallmatrix}2&-1 \\1&0\end{smallmatrix}) $ and
$\mathsf N|_{\FR}=(\begin{smallmatrix}2&1 \\-1&0\end{smallmatrix}) $, 
and the space--time mesh is taken as a Cartesian
mesh aligned with the space and time axes. The numerical fluxes are
defined through the following splitting of the matrix 
$\mathsf M=\left(\begin{smallmatrix}n^t_K&n^x_K \\n^x_K&n^t_K\end{smallmatrix}\right)$:
on the part of $\partial K$ with $n^x_K=1$,
$\mathsf M_+=\frac12(\begin{smallmatrix}1&1 \\1&1\end{smallmatrix})$, 
$\mathsf M_-=\frac12(\begin{smallmatrix}-1&1 \\1&-1\end{smallmatrix})$, 
while on the part of $\partial K$ with $n^x_K=-1$,
$\mathsf M_+=\frac12(\begin{smallmatrix}1&-1 \\-1&1\end{smallmatrix})$, 
$\mathsf M_-=\frac12(\begin{smallmatrix}-1&-1
  \\-1&-1\end{smallmatrix})$. 
With these definitions, the formulation of \cite{MoRi05} coincides with
\eqref{eq:VF} with
$\bV_p(\calT_h)=\IP^p(\calT_h)^2$ and $\alpha=\beta=\frac12$.
We remark that here we are interested in using different
approximating spaces.
Since we consider only meshes aligned with the axis, the semi-explicit
time-stepping based on macro elements described in \cite[section~3]{MoRi05} is not applicable in our setting.
However, the flux-splitting technique described there might be used to
generalise the DG fluxes~\eqref{eq:Fluxes} to non-Cartesian meshes.
\end{rem}

\begin{rem}
Once the discrete space $\bV_p(\calT_h)$ has been chosen, the
variational problem \eqref{eq:VF} can be solved as a global
algebraic linear system.
Alternatively, one could partition the time interval $[0,T]$ into
subintervals $[t_{j-1},t_{j}]$, $j=1,\ldots, n$, and solve
sequentially the formulation in every time slab $\Omega\times
[t_{j-1},t_{j}]$ using as initial conditions the traces of the
solution in the previous slab.
In this second case, the space--time mesh needs to be aligned with the
time slabs.
The two approaches are algebraically equivalent, but the second one 
{is in general}
computationally more advantageous.
In both cases the method is implicit in time. 
\end{rem}

\subsection{Trefftz-DG formulation}\label{ss:TDG}

In the following we restrict the discussion to the homogeneous initial value problem,
i.e.\ $J=0$ in \eqref{eq:IVP}. 
We define the Trefftz space 
$$\bT(\calT_h):=\bigg\{(\vE ,\vH )\in H^1(\calT_h)^2,\quad
\der{\vE }x+\der{(\mu \vH )}t=\der{\vH }x+\der{(\epsil \vE )}t=0 \quad
\text{in all}\;  K\in \calT_h \bigg\}.$$
If we choose $\bV_p(\calT_h)\subset \bT(\calT_h)$, the volume term in the
bilinear form \eqref{eq:Bilin} vanishes and the formulation reduces to
\begin{align}
&\text{seek}\;(\Ehp,\Hhp)\in \bV_p(\calT_h)\; \text{such that}\nonumber\\
&a_{TDG}(\Ehp,\Hhp;\vE ,\vH )=\ell_{TDG}(\vE ,\vH )\qquad \forall (\vE ,\vH )\in \bV_p(\calT_h),
\label{eq:VFTDG}
\end{align}
where
\begin{align}\label{eq:BilinTDG}
&\begin{aligned}
a_{TDG}(\Ehp,\Hhp;&\vE ,\vH ):=
\int_{\Fhor}(\epsil\Ehp^-\jmp{\vE }_t+\mu\Hhp^-\jmp{\vH }_t)\di x
+\int_\FT (\epsil\Ehp \vE +\mu\Hhp \vH )\di x\\
&+\int_\Fver \big( \mvl{\Ehp}\jmp{\vH }_x+\mvl{\Hhp}\jmp{\vE }_x
+\alpha\jmp{\Ehp}_x\jmp{\vE }_x+ \beta\jmp{\Hhp}_x\jmp{\vH }_x\big)\di t\\
&+\int_\FL \big(-\Hhp \vE +\alpha \Ehp \vE \big) \di t
+\int_\FR \big(\Hhp \vE +\alpha \Ehp \vE \big) \di t,
\end{aligned}
\\
&\begin{aligned}
\ell_{TDG}(\vE ,\vH ):=&
\int_\FO (\epsil E_0\vE  +\mu H_0 \vH )\di x\\&
+\int_\FL E_L(\vH +\alpha \vE )\di t
+\int_\FR E_R(-\vH +\alpha \vE )\di t.
\end{aligned}\nonumber
\end{align}

\begin{rem}
With the choice of $\bV_p(\calT_h)$ as in section~\ref{s:ConvRates} below, and with the choice
$\alpha=\beta=0$, the formulation \eqref{eq:VFTDG} coincides with the
method introduced in \cite{KSTW2014}.
\end{rem}

\begin{rem}\label{rem:HigherD}
In three space dimensions, Maxwell's equations with zero source term read:
\begin{align*}
\begin{aligned}
&\nabla\times\bE+\der{(\mu \bH)}t = \bzero  ,\qquad 
\nabla\times\bH-\der{(\epsil \bE)}t = \bzero &&\iin Q.
\end{aligned}
\end{align*}
Consider a space--time mesh whose elements are Cartesian products of Lipschitz polyhedra in space and time intervals.
In this case, the ``horizontal faces'' of the mesh skeleton are the polyhedra across which $t$ varies, the ``vertical faces'' are those across which $\bx$ varies.
%
%
We define the jumps 
$\jmp{\mathbf{v}}_t:=(\mathbf{v}^--\mathbf{v}^+)$  
on horizontal faces, 
$\jmp{\mathbf{v}}_{\bT}:=
\bn_{K_1}^{\bx}\times \mathbf{v}_{|_{K_1}}+\bn_{K_2}^{\bx}\times \mathbf{v}_{|_{K_2}}$
on vertical faces, and 
$\Fd:=\deO\times I$.
In the case of Dirichlet boundary conditions 
$\bn_\Omega^\bx\times\bE=\bn_\Omega^{\bx}\times\bg(\bx,t)$
on $\partial\Omega\times I$, with 
numerical fluxes similar to those in~\eqref{eq:Fluxes} 
(in particular with $\hbEhp =\mvl{\bEhp} - \beta\jmp{\bHhp}_\bT$ and 
$\hbHhp =\mvl{\bHhp} + \alpha\jmp{\bEhp}_\bT$ on $\Fver$), 
and with the obvious
definition of the finite dimensional Trefftz space $\bV_p(\calT_h)$,
the Trefftz-DG formulation reads
\begin{align}
&\text{seek}\;(\bEhp,\bHhp)\in \bV_p(\calT_h)\; \text{such that}\nonumber\\
&a_{TDG}^{\rm 3D}(\bEhp,\bHhp; \bvE,\bvH)=\ell_{TDG}^{\rm 3D} (\bvE,\bvH)\qquad \forall (\bvE,\bvH )\in \bV_p(\calT_h),
\label{Maxwell3D}
\end{align}
with
\begin{align*}
&a_{TDG}^{\rm 3D} (\bEhp,\bHhp; \bvE,\bvH):=
\int_{\Fhor}(\epsil\bEhp^-\cdot\jmp{\bvE }_t+\mu\bHhp^-\cdot\jmp{\bvH }_t)\di \bx\\
&\quad+\int_\FT (\epsil\bEhp \cdot \bvE +\mu\bHhp \cdot \bvH )\di \bx\\
&\quad+\int_\Fver \big( -\mvl{\bEhp}\cdot\jmp{\bvH }_{\bT}+\mvl{\bHhp}\cdot\jmp{\bvE }_{\bT}
+\alpha\jmp{\bEhp}_{\bT} \cdot\jmp{\bvE }_{\bT}+ \beta\jmp{\bHhp}_{\bT} \cdot\jmp{\bvH }_{\bT}\big)\di  S\\
&\quad+\int_\Fd \big(\bHhp +\alpha(\bn_\Omega^{\bx}\times \bEhp) \big) \cdot (\bn_\Omega^{\bx}\times\bvE) \di  S,\\
&\ell_{TDG}^{\rm 3D} (\vE ,\vH ):=
\int_\FO (\epsil \mathbf{E}_0 \cdot\bvE  +\mu \mathbf{H}_0 \cdot\bvH )\di \bx
+\int_\Fd ( \bn_\Omega^{\bx}\times\bg)\cdot \big(-\bvH +\alpha (\bn_\Omega^{\bx}\times\bvE) \big)\di  S,
\end{align*}
where $(\mathbf{E}_0, \mathbf{H}_0)$ is the initial datum 
(see also \cite{EKSTWtransparent}).

For the homogeneous acoustic wave equation $-\Delta U+c^{-2} \frac{\partial^2 U}{\partial t^2}=0$ in any space dimension, setting $\bsigma=-\nabla U$ and $v=\der{U}t$, we have
\begin{align*}
\begin{aligned}
&\nabla v+\der{\bsigma}t = \bzero , \qquad 
\nabla\cdot\bsigma+\frac{1}{c^2}\der{v}t =  0 &&\iin Q,
\end{aligned}
\end{align*}
In the case of Dirichlet boundary conditions with datum $v=g(\bx,t)$
on $\partial\Omega\times I$,
again with numerical fluxes similar to those
in~\eqref{eq:Fluxes}, the Trefftz-DG formulation reads
\begin{align}
&\text{seek}\;(\Vhp,\Shp)\in \bV_p(\calT_h)\; \text{such that}\nonumber\\
&a_{TDG}^{\rm wave}(\Vhp,\Shp; w ,\btau  )=\ell_{TDG}^{\rm wave} (w ,\btau )\qquad \forall (w ,\btau )\in \bV_p(\calT_h),
\label{WavesND}
\end{align}
with
\begin{align*}
&a_{TDG}^{\rm wave} (\Vhp,\Shp; w ,\btau ):=
\int_{\Fhor}\big(c^{-2}\Vhp^-
\jmp{w}_t 
+\Shp^-\cdot\jmp{\btau}_t \big) 
\di \bx
+\int_\FT (c^{-2}\Vhp  w +\Shp \cdot\btau )\di \bx\\
&\qquad\qquad+\int_\Fver \big( \mvl{\Vhp}\jmp{\btau }_{\bN}+\mvl{\Shp}\cdot\jmp{ w }_{\bN}
+\alpha\jmp{\Vhp}_{\bN}\cdot\jmp{ w }_{\bN}+ \beta\jmp{\Shp}_{\bN}\jmp{\btau }_{\bN}\big)\di S\\
&\qquad\qquad+\int_\Fd \big(\bsigma\cdot\bn_\Omega +\alpha \Vhp \big)   w \di S,
\\
&\ell_{TDG}^{\rm wave} ( w ,\btau ):=
\int_\FO ( c^{-2}v_0 w  +\bsigma_0\cdot \btau )\di \bx+\int_\Fd
g(\alpha  w -\btau\cdot\bn_\Omega)\di S,
\end{align*}
where $\bsigma_0=-\nabla U(\cdot,0)$ and $v_0=\frac{\partial
  U}{\partial t}(\cdot,0)$ are (given) initial data.
Here, the jumps are defined as follows: 
$\jmp{w}_t:=(w^--w^+)$ and 
$\jmp{\btau}_t:=(\btau^--\btau^+)$ on horizontal faces, 
$\jmp{w}_{\bN}:=w_{|_{K_1}}\bn_{K_1}^{\bx}+w_{|_{K_2}}
\bn_{K_2}^{\bx}$ and
$\jmp{\btau}_{\bN}:=\btau_{|_{K_1}}\cdot\bn_{K_1}^{\bx}+\btau_{|_{K_2}}
\cdot\bn_{K_2}^{\bx}$ on vertical faces.

The theoretical results proved in section~\ref{s:Analysis}
hold true also in these cases, {\em mutatis mutandis}, with very similar 
proofs.
\end{rem}

\section{Analysis of the Trefftz-DG method}\label{s:Analysis}

In this section we prove the well-posedness of the Trefftz-DG method
and its quasi-optimality in a mesh-dependent
norm (Theorem \ref{th:QuasiOpt}), as well as error estimates in
$L^2(Q)$ (Corollary \ref{cor:L2}).
The analysis is carried out within the framework developed in \cite{PVersion} for time-harmonic wave
problems.

\subsection{Well-posedness and quasi-optimality}\label{ss:WellPosed}

We define two seminorms on $H^1(\calT_h)^2$:
\begin{align}\label{eq:DGNorms}
&\begin{aligned}
\Tnorm{(\vE ,\vH )}^2_{DG}:=\;&
\frac12 \N{\epsil^{1/2}\jmp{\vE }_t}_{L^2(\Fhor)}^2
+\frac12\N{\mu^{1/2}\jmp{\vH }_t}_{L^2(\Fhor)}^2  \\
&+\frac12\N{\epsil^{1/2}\vE }^2_{L^2(\FO\cup\FT)}  
+\frac12\N{\mu^{1/2}\vH }^2_{L^2(\FO\cup\FT)} \\
&+\N{\alpha^{1/2}\jmp{\vE }_x}_{L^2(\Fver)}^2
+\N{\beta^{1/2}\jmp{\vH }_x}_{L^2(\Fver)}^2
+\N{\alpha^{1/2}\vE }_{L^2(\FL\cup\FR)}^2,
\end{aligned}
\\
&\begin{aligned}
\Tnorm{(\vE ,\vH )}^2_{DG^+}:=\;&
\Tnorm{(\vE ,\vH )}^2_{DG}
+\N{\epsil^{1/2}\vE^-}_{L^2(\Fhor)}^2
+\N{\mu^{1/2}\vH^-}_{L^2(\Fhor)}^2  \\
&+\N{\beta^{-1/2}\mvl{\vE }}_{L^2(\Fver)}^2
+\N{\alpha^{-1/2}\mvl{\vH }}_{L^2(\Fver)}^2\\&
+\N{\alpha^{-1/2}\vH }_{L^2(\FL\cup\FR)}^2.
\end{aligned}
\nonumber
\end{align}
The parameters $\epsil,\mu$ enter the DG and the DG${}^+$ norms only
through their traces on horizontal edges where they are continuous.

While $\Tnorm{\cdot}_{DG}$ is only a seminorm in $H^1(\calT_h)^2$, it defines a norm on $\bT(\calT_h)$.
\begin{lemma}\label{lem:DGisNorm}
The seminorm $\Tnorm{\cdot}_{DG}$ (and thus also $\Tnorm{\cdot}_{DG^+}$) is
a norm on the Trefftz space $\bT(\calT_h)$.
\end{lemma}
\begin{proof}
It suffices to verify that if $\Tnorm{(\vE ,\vH )}_{DG}=0$ then $\vE =\vH =0$.
If  $\Tnorm{(\vE ,\vH )}_{DG}=0$, then the pair $(\vE ,\vH )\in H^1(\Omega)^2$
and satisfies the initial value problem \eqref{eq:IVP} with $E_0=H_0=0$,
$E_L=E_R=0$ and $J=0$ (this last identity follows from $(\vE ,\vH )\in\bT(\calT_h)$).
Since problem~\eqref{eq:IVP} admits a unique solution {\cite[section~7.2.2c]{EVA02}}, then $\vE =\vH =0$.
\end{proof}

\begin{proposition}[Coercivity]\label{prop:Coercivity}
For all $(\vE ,\vH )\in H^1(\calT_h)^2$ the following identity holds
true:
\begin{align}\label{eq:Coercivity}
a_{DG}(\vE ,\vH ;\vE ,\vH )=\Tnorm{(\vE ,\vH )}^2_{DG}.
\end{align}
In particular, the bilinear form $a_{TDG}(\cdot\,;\cdot)$ is
coercive  in the space $\bT(\calT_h)$  with respect to the DG norm, with coercivity constant equal to 1.
\end{proposition}
\begin{proof}
Using the elementwise integration by parts in time and space
\begin{align}
\begin{aligned}
\sum_{K\in\calT_h}\iint_K
\der Ft\di x\di t &= 
\int_\Fhor \jmp{F}_t \di x
+\int_\FT F \di x
-\int_\FO F \di x\\
\sum_{K\in\calT_h}\iint_K
\der Fx\di x\di t &=
\int_\Fver \jmp{F}_x \di t
+\int_\FR F \di t
-\int_\FL F \di t
\qquad \forall F\in W^{1,1}(\calT_h),
\end{aligned}
\label{eq:basicIBP}
\end{align}
and the jump identity
\begin{align}\label{eq:JumpSquared}
v^-\jmp{v}_t-\frac12 \jmp{v^2}_t = \frac12 \jmp{v}^2_t
\qquad \oon \Fhor, \quad \forall v\in H^1(\calT_h),
\end{align}
we obtain the identity in the assertion:
\begin{align*}
a_{DG}(\vE ,\vH ;\vE ,\vH )\overset{\eqref{eq:Bilin}}=&
-\sum_{K\in\calT_h}\iint_K
\bigg(\frac12\der{}t\Big(\epsil\vE^2+\mu\vH^2\Big)+\der{}x(\vE \vH )
\bigg)\di x \di t\\
&+\int_{\Fhor}(\epsil\vE^-\jmp{\vE }_t+\mu\vH^-\jmp{\vH }_t)\di x
+\int_\FT (\epsil\vE^2 +\mu\vH^2 )\di x\\
&+\int_\Fver \big( \mvl{\vE}\jmp{\vH }_x+\mvl{\vH}\jmp{\vE }_x
+\alpha\jmp{\vE }_x^2+ \beta\jmp{\vH }_x^2\big)\di t\\
&+\int_\FL \big(-\vH \vE +\alpha \vE^2 \big) \di t
+\int_\FR \big(\vH\vE +\alpha \vE^2 \big) \di t
\\
\overset{\eqref{eq:basicIBP}}=&
\int_{\Fhor}\bigg(\epsil\vE^-\jmp{\vE }_t+\mu\vH^-\jmp{\vH }_t
-\frac12\epsil\jmp{\vE^2}_t-\frac12\mu\jmp{\vH^2}_t\bigg)\di x\\
&+\frac12\int_\FO(\epsil\vE^2+\mu\vH^2)\di x
+\frac12\int_\FT (\epsil\vE^2 +\mu\vH^2 )\di x\\
&+\int_\Fver\hspace{-1.8mm} \Big( \underbrace{\mvl{\vE}\jmp{\vH }_x+\mvl{\vH}\jmp{\vE }_x
-\jmp{\vE\vH}_x}_{=0}
+\alpha\jmp{\vE }_x^2+ \beta\jmp{\vH }_x^2 \Big)\di t\\
&+\int_\FL \alpha \vE^2 \di t
+\int_\FR \alpha \vE^2  \di t\\
\overset{\eqref{eq:DGNorms},\eqref{eq:JumpSquared}}=&\Tnorm{(\vE,\vH)}_{DG}^2.
\end{align*}
The coercivity of $a_{TDG}(\cdot,\cdot)=a_{DG}(\cdot,\cdot)$ in $\bT(\calT_h)$ follows from Lemma \ref{lem:DGisNorm}.
\end{proof}

\begin{proposition}[Continuity]\label{prop:Continuity}
The following continuity bound holds:
\begin{align*}
\abs{a_{TDG}(E,H;\vE,\vH)}
\le  2
\:\Tnorm{(E,H)}_{DG^+} \Tnorm{(\vE,\vH)}_{DG} 
\qquad \forall (E,H),(\vE,\vH)\in \bT(\calT_h).
\end{align*}
Moreover, when $E_L=E_R=0$, it holds that
\[
\abs{\ell_{TDG}(v_E,v_H)}\le \sqrt{2}\,\Big(
\N{\epsil^{1/2}E_0}^2_{L^2(\FO)}+\N{\mu^{1/2}H_0}^2_{L^2(\FO)}\Big)^{1/2}\Tnorm{(v_E,v_H)}_{DG}.
\]
\end{proposition}
\begin{proof}
The assertions follow from the definition of the bilinear form and
of the linear functional in \eqref{eq:BilinTDG}, the norms in \eqref{eq:DGNorms}, and the
Cauchy--Schwarz inequality.
\end{proof}

\begin{theorem}[Quasi-optimality]\label{th:QuasiOpt}
For any finite dimensional $\bV_p(\calT_h)\subset \bT(\calT_h)$,
the Trefftz-DG formulation~\eqref{eq:VFTDG} admits
a unique  solution $(\Ehp,\Hhp)\in\bV_p(\calT_h)$.
Moreover, the following quasi-optimality bound holds:
\begin{align}\label{eq:QuasiOpt}
\Tnorm{(E,H)-(\Ehp,\Hhp)}_{DG}
\le  3 
\inf_{(\vE,\vH)\in\bV_p\Th}
\Tnorm{(E,H)-(\vE,\vH)}_{DG^+}.
\end{align}
\end{theorem}
\begin{proof}
To prove uniqueness, assume that $E_L=E_R=E_0=H_0=0$. 
Proposition~\ref{prop:Coercivity} implies $\Ehp=\Hhp=0$. 
Existence follows from uniqueness.
For~\eqref{eq:QuasiOpt}, the triangle inequality gives
\begin{equation}\label{eq:tria}
\Tnorm{(E,H)-(\Ehp,\Hhp)}_{DG}\le \Tnorm{(E,H)-(v_E,v_H)}_{DG}+\Tnorm{(\Ehp,\Hhp)-(v_E,v_H)}_{DG}
\end{equation}
for all $(v_E,v_H)\in \bV_p(\calT_h)$. Since $(\Ehp,\Hhp)-(v_E,v_H)\in
\bV_p(\calT_h)\subset \bT(\calT_h)$,
Proposition~\ref{prop:Coercivity},
consistency (which follows by construction and from the consistency of the
numerical fluxes), and Proposition~\ref{prop:Continuity} give
\[
\begin{split}
\Tnorm{(\Ehp,\Hhp)-(v_E,v_H)}_{DG}^2=&
\,a_{TDG}(E-\vE,H-\vH; \Ehp-\vE,\Hhp-\vH\big)\\
\le&\,  2 \,\Tnorm{(E,H)-(v_E,v_H)}_{DG^+}\Tnorm{(\Ehp,\Hhp)-(v_E,v_H)}_{DG},
\end{split}
\]
which, together with~\eqref{eq:tria}, implies~\eqref{eq:QuasiOpt}.
\end{proof}

\begin{rem}\label{rem:ContinuousDep}
If $E_L=E_R=0$, i.e.\ the lateral boundary conditions are homogeneous, then the right-hand
side functional $\ell_{TDG}(\cdot)$ is continuous in DG norm, see  Proposition~\ref{prop:Continuity}.
This, together with Proposition~\ref{prop:Coercivity}, immediately gives
a stability bound on the discrete solutions in terms of the data, i.e.\ 
$$\Tnorm{(\Ehp,\Hhp)}_{DG}\le 
\sqrt2\,\Big(\N{\epsil^{1/2}E_0}^2_{L^2(\FO)}+\N{\mu^{1/2}H_0}^2_{L^2(\FO)}\Big)^{1/2}.$$
Otherwise, we only have stability in terms of the exact solution from \eqref{eq:QuasiOpt}:
$$\Tnorm{(\Ehp,\Hhp)}_{DG}\le  4 
\Tnorm{(E,H)}_{DG^+}.$$
The reason why a stability bound in terms of $E_0,H_0,E_L,E_R$ does not hold if $E_L,E_R\ne 0$ is that, in this case, the integrals on $\FL$ and $\FR$ in the definition of $\ell_{TDG}(\cdot)$ (see~\eqref{eq:BilinTDG}) do not vanish and bounding them by the Cauchy--Schwarz inequality generates
terms with $\vH$ on $\FL\cup\FR$,
which only allow to bound $\abs{\ell_{TDG}(v_E,v_H)}$ from above
with $\Tnorm{(v_E,v_H)}_{DG^+}$, instead of the weaker norm $\Tnorm{(v_E,v_H)}_{DG}$:
\[
\begin{split}
\abs{\ell_{TDG}(v_E,v_H)}\le \sqrt{2}\,\Big(&
\N{\epsil^{1/2}E_0}_{L^2(\FO)}^2+\N{\mu^{1/2}H_0}^2_{L^2(\FO)}\\
&+\N{\alpha^{1/2}E_L}^2_{L^2(\FL)}+\N{\alpha^{1/2}E_R}^2_{L^2(\FR)}
\Big)^{1/2}\cdot\Tnorm{(v_E,v_H)}_{DG^+}.
\end{split}
\]
\end{rem}

\begin{rem}\label{rem:WeightedAvg}
Let us fix $0\le \xi\le 1$.
In the definition \eqref{eq:Fluxes} of the numerical fluxes on $\Fver$ 
we may substitute to the averages $\mvl{\Ehp}$ and $\mvl{\Hhp}$ the weighted averages 
$\mvl{\Ehp}_{\xi}$ and $\mvl{\Hhp}_{1-\xi}$ respectively, where
we have set $\mvl{\phi}_\xi:=\xi\phi_{|K_1}+(1-\xi)\phi_{|K_2}$ on $(\partial K_1\cap\partial K_2)\subset \Fver$.
All the results obtained in this and the following sections remain valid
in this case.
\end{rem}

\subsection{Estimates in \texorpdfstring{$L^2(Q)$}{L2(Q)}~norm}\label{ss:L2}

By virtue of Theorem \ref{th:QuasiOpt}, we can control the Trefftz-DG
error in DG~norm; it is of course desirable to prove a bound on the
error measured in a mesh-independent norm.
Following the argument developed for time-harmonic problems in \cite[Theorem~3.1]{MOW99} (see also \cite[Theorem~4.1]{BUM07}, 
\cite[Lemma~3.7]{PVersion}), in Proposition~\ref{prop:Duality} we
prove that the $L^2(Q)$~norm of any Trefftz function is bounded by its
DG~norm, thus the error estimate in $L^2(Q)$~norm readily follows, see Corollary~\ref{cor:L2}.

The application of the technique of \cite[Theorem~3.1]{MOW99} relies on a certain
stability estimate for the following auxiliary problem:
\begin{align}\label{eq:dualIVP}
\begin{aligned}
&\der{\vE}x+\der{(\mu \vH)}t = \phi &&\iin Q,\\
&\der{\vH}x+\der{(\epsil \vE)}t = \psi &&\iin Q,\\
&\vE(\cdot,0)=0, \quad
\vH(\cdot,0)=0 &&\oon \Omega,\\
&\vE(x_L,\cdot)=0,\quad
\vE(x_R,\cdot)=0 &&\oon I,
\end{aligned}
\end{align}
for $\phi,\psi\in L^2(Q)$. 
More precisely, we will need a bound on  the $L^2$~norm of the traces of $\vE$ and $\vH$ on horizontal and vertical segments in terms of the $L^2(Q)$~norm of $(\phi,\psi)$:
\begin{align}\nonumber
&\N{\epsil^{1/2}\vE}^{2}_{L^2(\Fhor\cup{\FT})}
+\N{\mu^{1/2}\vH}^{2}_{L^2(\Fhor\cup{\FT})}\\&	
+\N{{\beta^{-1/2}}\vE}^{2}_{L^2(\Fver)}
+\N{{\alpha^{-1/2}}\vH}^{2}_{L^2(\Fver\cup\FL\cup\FR)}\nonumber\\
&\qquad\le \Cstab^{2}
\bigg(\N{\epsil^{1/2}\phi}^{2}_{L^2(Q)}+\N{\mu^{1/2}\psi}^{2}_{L^2(Q)}\bigg)
\hspace{15mm} \forall (\phi,\psi)\in L^2(Q)^2,
\label{eq:DualStability}
\end{align}
for some $\Cstab>0$. We have inserted the numerical flux parameters within
the
third and fourth term on the left-hand side of~\eqref{eq:DualStability} because this is what
we need in the proof of Proposition~\ref{prop:Duality} below; then the
constant $\Cstab$ will also depend on $\alpha$ and
$\beta$. 


\begin{prop} \label{prop:Duality}
Assume that the estimate \eqref{eq:DualStability} holds true for
$(\vE,\vH)$ solution of problem~\eqref{eq:dualIVP}.
Then, for any Trefftz function $(\wE,\wH)\in \bT\Th$, the $L^2(Q)$~norm is bounded by the DG norm:
\begin{align*}
\Big(\N{\mu^{-1/2}\wE}_{L^2(Q)}^2+\N{\epsil^{-1/2}\wH}_{L^2(Q)}^2\Big)^{1/2}
\le  \sqrt2\, \Cstab\,\Tnorm{(\wE,\wH)}_{DG},
\end{align*}
with $\Cstab$ as in \eqref{eq:DualStability}.
\end{prop}
\begin{proof}
Let $(\vE,\vH)$ be the solution of the auxiliary problem \eqref{eq:dualIVP}. 
The space--time vector field $(\vE,\mu\vH)$ belongs to
$H(\dive_{x,t};Q)$, thus it has vanishing normal jumps across any smooth curve lying in the interior of $Q$; in particular $\jmp{\vE}_x=\jmp{\mu\vH}_t=0$.
Similarly, $(\vH,\epsil\vE)\in H(\dive_{x,t};Q)$ implies $\jmp{\vH}_x=\jmp{\epsil\vE}_t=0$.
Multiplying the functions to be bounded with the source terms of
problem  \eqref{eq:dualIVP} and integrating by parts over the
elements, we have
\begin{align*}
&\hspace{-10mm}\iint_Q(\wE\psi+\wH\phi)\di x\di t\\
=&\sum_{K\in\calT_h}\iint_K \bigg(\wE\der{\vH}x+\wE\der{(\epsil \vE)}t
+\wH\der{\vE}x+\wH\der{(\mu \vH)}t\bigg)\di x\di t\\
=&-\sum_{K\in\calT_h}\Bigg(\iint_K \bigg(\underbrace{\der{(\epsil \wE)}t\vE+\der{\wH}x\vE}_{=0}
 +\underbrace{\der{\wE}x\vH+\der{(\mu \wH)}t\vH}_{=0}\bigg)\di x\di t\\
&+\int_{\partial K}(\epsil\wE\vE n_K^t+\mu\wH\vH n_K^t+\wE\vH n_K^x+\wH\vE n_K^x)\di s\Bigg)\\
=&
\int_\Fhor\underbrace{\jmp{\epsil\wE\vE+\mu\wH\vH}_t}_{=\epsil\jmp{\wE}_t\vE
+\mu\jmp{\wH}_t\vH}\di x\\
&
 +\int_\FT(\epsil\wE\vE  +\mu\wH\vH) \di x
 -\int_\FO(\epsil\wE{\underbrace{\vE}_{=0}} +\mu\wH{\underbrace{\vH}_{=0}}) \di x\\
&+\int_\Fver \underbrace{\jmp{\wE\vH+\wH\vE}_x}_{=\jmp{\wE}_x\vH+\jmp{\wH}\vE}\di t\\
&-\int_\FL (\wE\vH+\wH\underbrace\vE_{=0})\di t
+\int_\FR (\wE\vH+\wH\underbrace\vE_{=0})\di t\\
\le&\Tnorm{(\wE,\wH)}_{DG}\\
&\cdot\bigg(2\int_{\Fhor\cup{\FT}}(\epsil\vE^2+\mu\vH^2)\di x
+\int_{\Fver}(\beta^{-1}\vE^2+\alpha^{-1}\vH^2) \di t
+\int_{\FL\cup\FR}\alpha^{-1}\vH^2 \di t\bigg)^{1/2}\\
\overset{\eqref{eq:DualStability}}\le& \sqrt2\,\Cstab\Tnorm{(\wE,\wH)}_{DG}  \bigg(\N{{\epsil^{1/2}}\phi}^{2}_{L^2(Q)}+\N{{\mu^{1/2}}\psi}^{2}_{L^2(Q)}\bigg)^{1/2}.
\end{align*}
Since
\[
\Big(\N{\mu^{-1/2}\wE}_{L^2(Q)}^2+\N{\epsil^{-1/2}\wH}_{L^2(Q)}^2\Big)^{1/2}=\!
\sup_{(\phi,\psi)\in L^2(Q)^2}\frac{\iint_Q(\wE\psi+\wH\phi)\di x\di t}{\big(\N{{\epsil^{1/2}}\phi}^{2}_{L^2(Q)}+\N{{\mu^{1/2}}\psi}^{2}_{L^2(Q)}\big)^{1/2}},
\]
we obtain the desired estimate.
\end{proof}

Recalling that the error $((E-\Ehp),(H-\Hhp))\in\bT\Th$, and combining Proposition~\ref{prop:Duality} and the quasi-optimality in
DG~norm proved in Theorem~\ref{th:QuasiOpt}, we obtain the following
bound on the Trefftz-DG error measured in $L^2(Q)$~norm.

\begin{cor}[Quasi-optimality in $L^2(Q)$]\label{cor:L2}
Under the assumptions of Proposition \ref{prop:Duality}, for any finite dimensional Trefftz space $\bV_p(\calT_h)\subset\bT\Th$, the solution
$(\Ehp,\Hhp)$ of the Trefftz-DG formulation~\eqref{eq:VFTDG} satisfies
the bound
\begin{align*}
\Big(\N{\mu^{-1/2}(E-\Ehp)}_{L^2(Q)}^2
&+\N{\epsil^{-1/2}(H-\Hhp)}_{L^2(Q)}^2\Big)^{1/2}\\
&\le 3\sqrt2\, \Cstab
\inf_{(\vE,\vH)\in\bV_p\Th}\Tnorm{(E,H)-(\vE,\vH)}_{DG^+},
\end{align*}
with $\Cstab$ as in \eqref{eq:DualStability}.
\end{cor}

We prove now the stability bound \eqref{eq:DualStability} for the
  solution of the initial auxiliary problem \eqref{eq:dualIVP}, with
  additional assumptions on the flux parameters $\alpha$ and $\beta$.
The proof is based on differentiation of an energy functional
and Gronwall's lemma.
In case of constant material parameters $\epsil$ and $\mu$, one can
derive the bound
\eqref{eq:DualStability}, based on an exact representation
of the solution of~\eqref{eq:dualIVP},
with a better constant $\Cstab$ than that of
Lemma~\ref{lem:DualStability}, with no additional assumptions on $\alpha$ and $\beta$, see appendix \ref{appendix}; the generalisation of this argument to higher space
dimensions and general geometries is not straightforward.

We introduce some notation. For an element $K\in\calT_h$, we denote by
$h_K^x$ its horizontal edge length and set $h^x:=\max_{K\in\calT_h}h_K^x$. 
For a face $f\subset\Fver$, $f=\partial K_1\cup\partial K_2$, we define
\[
h_f^x:=\min\{h_{K_1}^x,h_{K_2}^x\}, \qquad
\epsil_f:=\max\{\epsil_{K_1},\epsil_{K_2}\},\qquad
\mu_f:=\max\{\mu_{K_1},\mu_{K_2}\},
\]
while for a face $f\subset\FL\cup\FR$, $f\subset \partial K$,
\[
h_f:=h_K^x,\qquad 
\epsil_f:=\epsil_K,\qquad
\mu_f:=\mu_K.
\]

\begin{lemma}\label{lem:DualStability}
Assume that the flux parameters $\alpha$ and $\beta$ have the
following expressions on any face $f\subset\Fver\cup\FL\cup\FR$:
\begin{equation}\label{eq:alphabeta}
\alpha_{|_f}={\tt a}\,\frac{h^x}{h_f^x}\,\epsil_f, \qquad
\beta_{|_f}={\tt b}\,\frac{h^x}{h_f^x}\,\mu_f,
\end{equation}
where ${\tt a}$ and ${\tt b}$ are positive constants independent of the
mesh size, the material coefficients, and the local approximating spaces.

The solution $(\vE,\vH)$ of the initial auxiliary problem \eqref{eq:dualIVP} satisfies the stability bound 
\eqref{eq:DualStability} with 
\begin{equation}\label{eq:Cdual}
\Cstab^2\le (N_{\mathrm{hor}}\, e^T c^2_\infty) +
\min\{{\tt a}^{-1},{\tt b}^{-1}\}
\bigg(\frac{4T^2}{h^x} c_\infty^4 + \frac{h^x}2 c_\infty^2+ 2c_\infty^3 N_{\mathrm{hor}} e^T\bigg),
\end{equation}
where we have set
$N_{\mathrm{hor}}:=\#\big\{t, \text{ such that } (x,t)\in\Fhor\cup\FT 
\text{ for some }x_L<x<x_R\big\}$ and $c_\infty:=\N{c}_{L^\infty(Q)}$.

\end{lemma}
\begin{proof}
We assume that $\phi$ and $\psi$ are continuous and compactly supported in $Q$; the general case will follow by a density argument.

Let $(\vE,\vH)$ be the solution of problem
\eqref{eq:dualIVP}, and set 
\begin{equation}\label{EnergyDef}
\calE(t):=\frac12 \int_{\Omega\times\{t\}}(\epsil\vE^2+\mu\vH^2)\di
x. 
\end{equation}
The initial conditions give
$\calE(0)=0$, while the equations and the boundary conditions imply
\[
\begin{split}
\der{}t \calE(t)
&=\int_{\Omega\times\{t\}}\Big(\vE \der{(\epsil\vE)} t+\vH \der{(\mu\vH)} t\Big)\di x
=\int_{\Omega\times\{t\}}\Big(-\der{}x(\vE \vH)+ \vE\psi +\vH \phi\Big)\di x\\
&=-\underbrace{\vE(x_R,t)}_{=0}\vH(x_R,t) + \underbrace{\vE(x_L,t)}_{=0}\vH(x_L,t)
+\int_{\Omega\times\{t\}}\Big(\vE\psi +\vH \phi\Big)\di x,
\end{split}
\]
which in turns implies
\begin{equation}\label{eq:EnergyId}
\calE(t)=\underbrace{\calE(0)}_{=0}+\iint_{\Omega\times(0,t)}\Big(\vE\psi +\vH \phi\Big)\di x\di s,
\end{equation}
and
\begin{align}
\der{}t \calE(t)
&\le\frac12
\int_{\Omega\times\{t\}}\Big( \epsil\vE^2+\epsil^{-1}\psi^2 +\mu\vH^2+
\mu^{-1}\phi^2\Big)\di x\ 
\nonumber\\
&=\calE(t)+\frac12\N{\epsil^{-1/2}\psi}^2_{L^2(\Omega\times\{t\})}
+\frac12\N{\mu^{-1/2}\phi}^2_{L^2(\Omega\times\{t\})}.
\label{eq:EnergyIneq}
\end{align}
%
%
Gronwall's lemma as in \cite[p.~624]{EVA02},
$\eta'(t)\le a(t)\eta(t)+b(t)\Rightarrow \eta(t)\le e^{\int_0^t
  a(s)\di s}(\eta(0)+\int_0^t b(s)\di s)$,
applied to~\eqref{eq:EnergyIneq} gives 
\begin{align}\label{eq:EnergySource}
\calE(t)\le e^t\Big(
\underbrace{\calE(0)}_{=0}+\frac12\N{\mu^{-1/2}\phi}^2_{L^2(\Omega\times(0,t))}
+\frac12\N{\epsil^{-1/2}\psi}^2_{L^2(\Omega\times(0,t))}
\Big).
\end{align}
Taking into account the definition of $\calE(t)$ in~\eqref{EnergyDef}, the bound \eqref{eq:EnergySource} allows to control the terms on horizontal faces.
We denote by $T_{\mathrm{hor}}$ the set 
$\big\{t\in(0,T], \text{ such that } (x,t)\in\Fhor\cup\FT\text{ for some }x_L<x<x_R\big\}$;
recall that $N_{\mathrm{hor}}=\#T_{\mathrm{hor}}$.
Using $1/\epsil\mu=c^2$, we have
\begin{align}
&\N{\epsil^{1/2}\vE}^{2}_{L^2(\Fhor\cup{\FT})}
+\N{\mu^{1/2}\vH}^{2}_{L^2(\Fhor\cup{\FT})}
\nonumber\\
&\qquad\qquad\qquad
\le \sum_{t_j\in T_{\mathrm{hor}}} e^{t_j}
\bigg(\N{\mu^{-1/2}\phi}^2_{L^2(\Omega\times(0,t_j))}+\N{\epsil^{-1/2}\psi}^2_{L^2(\Omega\times(0,t_j))}
\bigg)
\label{eq:HorizBound}\\
&\qquad\qquad\qquad
\le 
(N_{\mathrm{hor}}\, e^T c_\infty^2)
\bigg(\N{\epsil^{1/2}\phi}^2_{L^2(Q)}+\N{\mu^{1/2}\psi}^2_{L^2(Q)} \bigg).
\nonumber
\end{align}

%
Integrating~\eqref{EnergyDef} in $(0,t)$ gives
\begin{align*}
\N{\epsil^{1/2}\vE}^{2}_{L^2(\Omega\times(0,t))}
&+\N{\mu^{1/2}\vH}^{2}_{L^2(\Omega\times(0,t))}
=2\int_0^t \calE(s) \di s\\
&\overset{\eqref{eq:EnergyId}}=2\int_0^t
\Big(
\iint_{\Omega\times(0,s)} (\vE\psi+\vH\phi)  \di x\di r\Big) \di s\\
&\le t\Big(\N{\epsil^{1/2}\vE}^{2}_{L^2(\Omega\times(0,t))}
+\N{\mu^{1/2}\vH}^{2}_{L^2(\Omega\times(0,t))}\Big)^{1/2}\\
&\qquad\cdot
\Big(\N{c\epsil^{1/2}\phi}^{2}_{L^2(\Omega\times(0,t))}
+\N{c\mu^{1/2}\psi}^{2}_{L^2(\Omega\times(0,t))}\Big)^{1/2},
\end{align*}
which implies
\begin{align}\label{eq:L2Q}
\N{\epsil^{1/2}\vE}^{2}_{L^2(\Omega\times(0,t))}
\!+\N{\mu^{1/2}\vH}^{2}_{L^2(\Omega\times(0,t))}
\le t^2 c_\infty^2
\Big(\N{\epsil^{1/2}\phi}^{2}_{L^2(\Omega\times(0,t))}
\!+\N{\mu^{1/2}\psi}^{2}_{L^2(\Omega\times(0,t))}\Big).
\end{align}

We proceed now by bounding the terms on vertical edges.
We denote by $x_K$ the midpoint of $K$ in $x$ direction, 
so that $2(x-x_K)\le h^x$ for all $(x,t)\in K$, by $\partial K^{\mathrm{WE}}$ and $\partial K^{\mathrm{SN}}$ the union of the 
vertical and horizontal edges of $K$, respectively.
Taking into account the expression of $\alpha$ and $\beta$
in~\eqref{eq:alphabeta} and defining for brevity
$A:=\min\{{\tt a}^{-1},{\tt b}^{-1}\}$,
we have
\begin{align*}
&\N{\beta^{-1/2}\vE}^{2}_{L^2(\partial K^{\mathrm{WE}})}
+\N{\alpha^{-1/2}\vH}^{2}_{L^2(\partial K^{\mathrm{WE}})}\\
&\le A\frac{h_K^x}{h^x}\left(
\N{\mu^{-1/2}\vE}^{2}_{L^2(\partial K^{\mathrm{WE}})}
+\N{\epsil^{-1/2}\vH}^{2}_{L^2(\partial K^{\mathrm{WE}})}
\right)\\
&=A\frac{h_K^x}{h^x}\frac2{h_K^x}\iint_K \frac{\partial}{\partial x}
\bigg((x-x_K)\mu^{-1}\vE^2+(x-x_K) \epsil^{-1}\vH^2 \bigg)\di x\di t\\
&=\frac{2A}{h^x}\iint_K \bigg(\mu^{-1}\vE^2+\epsil^{-1}\vH^2
+2(x-x_K)\Big(\mu^{-1} \vE\der \vE x+ \epsil^{-1}\vH\der\vH x\Big)\bigg)\di x\di t
\\
&\overset{\eqref{eq:dualIVP}}=\frac{2A}{h^x}\iint_K \bigg(\mu^{-1}\vE^2+\epsil^{-1}\vH^2
+2(x-x_K)\Big(-\der{}t(\vE\vH)+ \mu^{-1}\vE \phi+ \epsil^{-1}\vH\psi\Big)\bigg)\di x\di t
\\
&\le\frac{2A}{h^x}\iint_K \bigg(2\mu^{-1}\vE^2+2\epsil^{-1}\vH^2
+(x-x_K)^2\Big( \mu^{-1}\phi^2+ \epsil^{-1}\psi^2\Big)\bigg)\di x\di t\\
&\qquad+\frac{2A}{h^x}\int_{\partial K^{\mathrm{SN}}}2\abs{(x-x_K) \vE\vH}\di x
\end{align*}
Using $2\vE\vH\le (\xi\vE^2+\frac1\xi\vH^2)$ with
weight $\xi=\epsil c =(\mu c)^{-1}$, 
we have the bound
\begin{align*}
&\sum_{K\in\calT_h}\bigg(\N{\beta^{-1/2}\vE}^{2}_{L^2(\partial K^{\mathrm{WE}})}
+\N{\alpha^{-1/2}\vH}^{2}_{L^2(\partial K^{\mathrm{WE}})}\bigg)
\\
&= \frac{4A}{h^x}\N{\mu^{-1/2}\vE}^2_{L^2(Q)}+\frac{4A}{h^x}\N{\epsil^{-1/2}\vH}^2_{L^2(Q)}
+\frac{A h^x}2\N{\mu^{-1/2}\phi}^2_{L^2(Q)}+\frac{A h^x}2\N{\epsil^{-1/2}\psi}^2_{L^2(Q)}\\
&\qquad+2A\int_{\Fhor\cup\FT}c(\epsil\vE^2+\mu\vH^2)\di x
\\
&\overset{\eqref{eq:L2Q},\eqref{eq:HorizBound}}\le
A\bigg(\frac{4T^2}{h^x} c_\infty^4 + \frac{h^x}2 c_\infty^2+ 2c_\infty^3 N_{\mathrm{hor}} e^T\bigg)
\bigg(\N{\epsil^{1/2}\phi}^2_{L^2(Q)}+\N{\mu^{1/2}\psi}^2_{L^2(Q)}\bigg),
\end{align*}
recalling that $c_\infty=\N{c}_{L^\infty(Q)}$.
%

This, together with \eqref{eq:HorizBound}, gives the
bound~\eqref{eq:DualStability} with constant $\Cstab^2$ as in~\eqref{eq:Cdual}.
\end{proof}

In case of a tensor product mesh with all elements having horizontal edges of length $h^x$ and vertical edges of length $h^t=h^x/c$, 
the constant $\Cstab$ is proportional to $(h^x)^{-1/2}$.
We stress that we cannot expect a bound like \eqref{eq:DualStability} with $\Cstab$ independent of the meshwidth: 
indeed if the mesh is refined, say, uniformly, while the term in the brackets in the right-hand side of \eqref{eq:DualStability} is not modified, the left-hand side grows (consider e.g.\ the simple case $\phi=\mu$, $\psi=0$, $\vE=0$, $\vH=t$).

One could attempt to derive the stability bound
\eqref{eq:DualStability} by controlling with $(\phi,\psi)$ either
the $H^1(Q)$ or the $L^\infty(Q)$ norm of $(\vE,\vH)$, since both
these norms would then control the desired mesh-skeleton norm.
However, this is not possible, as the solution of problem
\eqref{eq:dualIVP} is in general not bounded in those norms.
Consider for example the simple case with source 
$\phi=\psi=\sqrt\ell\chi_{\{0<x-t<\ell^{-1}\}}$ for $\ell\in\IN$, in $Q=(-3,3)\times(0,1)$ and with $\epsil=\mu=1$.
The source is uniformly bounded in $L^2(Q)$
with respect to $\ell$, namely
$\N{\phi}^2_{L^2(Q)}=\N{\psi}^2_{L^2(Q)}=1$, 
but the solution
$\vE=\vH=t\sqrt\ell\chi_{\{0<x-t<\ell^{-1}\}}$ has a jump
across 
$x=t$, so it does not belong to $H^1(Q)$, and
$\N{\vE}_{L^\infty(Q)}=\N{\vH}_{L^\infty(Q)}=\sqrt\ell$ is not
uniformly bounded with respect to $\ell\in\IN$.

\subsection{Energy considerations}\label{ss:Energy}

Define the continuous and discrete energies at a given time $t\in[0,T]$:
\[
\calE(t):=\frac12\int_{\Omega\times\{t\}}(\epsil E^2+\mu H^2)\di x,\qquad
\calE_{hp}(t):=\frac12\int_{\Omega\times\{t\}}(\epsil \Ehp^2+\mu
\Hhp^2)\di x.
\]

Consider the case where $E_L=E_R=0$ and $J=0$; we have that the energy is preserved for
  the continuous problem. 

In fact, proceeding like in the first step of the proof of
  Lemma~\ref{lem:DualStability},
taking into account the equations in~\eqref{eq:IVP},
together with $E_L=E_R=0$ and $J=0$, we have $\frac{\partial}{\partial
  t}\calE(t)=0$,
which implies $\calE(t)=\calE(0)$ for any
$t>0$.


We turn now to the discrete case. In order to compute
  $\calE_{hp}(T)=\frac{1}{2}\int_\FT(\epsil \Ehp^2+\mu
\Hhp^2)\di x$, we consider the identity
$$0\overset{\eqref{eq:VFTDG}}=
\ell_{TDG}(\Ehp,\Hhp)-a_{TDG}(\Ehp,\Hhp;\Ehp,\Hhp)
\overset{\eqref{eq:Coercivity}}=\ell_{TDG}(\Ehp,\Hhp)-\Tnorm{(\Ehp,\Hhp)}_{DG}^2,$$
 and obtain, by simply expanding both terms and moving to the left-hand side the term on~$\FT$,
\begin{align}\label{eq:DiscreteEnergyEvo}
\calE_{hp}(T) =\: & \frac12\int_\FO(\epsil E_0^2+\mu H_0^2)\di x
-\frac12\int_\FO\big(\epsil (\Ehp-E_0)^2+\mu (\Hhp-H_0)^2\big)\di x
\\
&-\frac12\int_\Fhor (\epsil\jmp{\Ehp}^2_t+\mu\jmp{\Hhp}_t^2)\di x
-\int_\Fver (\alpha\jmp{\Ehp}^2_x+\beta\jmp{\Hhp}_x^2)\di t
\nonumber\\
&+ \int_\FL \big(\alpha\Ehp(E_L-\Ehp)+E_L\Hhp\big)\di t
+ \int_\FR \big(\alpha\Ehp(E_R-\Ehp)-E_R\Hhp\big)\di t.
\nonumber
\end{align}
Observing the signs of the terms in this expression, we note that in the case $E_L=E_R=0$ the method is dissipative:
$\calE_{hp}(T) \le \frac12\int_\FO(\epsil E_0^2+\mu H_0^2)\di x$.


\subsection{The 
problem with Robin boundary conditions}\label{ss:Robin}

So far we have studied the initial boundary value problem
\eqref{eq:IVP} with
Dirichlet boundary conditions only ($E=E_{L/R}$ on $\calF_h^{L/R}$). 
In the case of the problem \eqref{eq:RobinIVP} with Robin boundary conditions ($\epsil^{1/2}E-\mu^{1/2}Hn_Q^x=g_{L/R}$ on $\calF_h^{L/R}$) the formulation and the
analysis of the Trefftz-DG scheme follow in the same way.
We outline in this section the differences.
We denote the modified quantities with the superscript $\calR$.

We fix a new flux parameter $\delta\in L^\infty(\FL\cup\FR)$ 
satisfying $0<\delta_*\le \delta\le \delta^*<1$.
The numerical fluxes \eqref{eq:Fluxes} on $\FL$ and $\FR$ are modified as 
\begin{align*}
&\hEhp^\calR=\begin{cases}
\Ehp-\delta\big(\Ehp+(\mu/\epsil)^{1/2}\Hhp-\epsil^{-1/2}g_L\big)
&\oon \FL,\\
\Ehp-\delta\big(\Ehp-(\mu/\epsil)^{1/2}\Hhp-\epsil^{-1/2}g_R\big)
&\oon \FR,
\end{cases}\\
&\hHhp^\calR=\begin{cases}
\Hhp+(1-\delta)\big(-(\epsil/\mu)^{1/2}\Ehp-\Hhp+\mu^{-1/2}g_L\big)
&\oon \FL,\\
\Hhp+(1-\delta)\big((\epsil/\mu)^{1/2}\Ehp-\Hhp-\mu^{-1/2}g_R\big)
&\oon \FR.
\end{cases}
\end{align*}
This choice arises from imposing consistency (i.e.\ for $\Ehp=E$ and
$\Hhp=H$ we recover $\widehat E^\calR=E$ and $\widehat H^\calR=H$), and
imposing
that the fluxes themselves satisfy the boundary condition
(i.e.\ $\epsil^{1/2}\hEhp^\calR-\mu^{1/2}\hHhp^\calR n_Q^x=g_{L/R}$).
We note that now the parameter $\alpha$ needs to be defined on $\Fver$
only (as opposed to on $\Fver\cup\FL\cup\FR$ in the case of
  Dirichlet boundary conditions).

The bilinear forms $a_{DG}(\cdot,\cdot)$ in \eqref{eq:Bilin} 
and $a_{TDG}(\cdot,\cdot)$ in \eqref{eq:BilinTDG}
are modified only in the terms on $\FL$ and $\FR$;
for example, $a_{TDG}(\cdot,\cdot)$ becomes
\begin{align*}
a_{TDG}^\calR(&\Ehp,\Hhp;\vE,\vH)\\
=\ldots\;&+
\int_\FL\Big(-(1-\delta)\Ehp \vH+\delta(\mu/\epsil)^{1/2}\Hhp\vH 
-\delta\Hhp\vE+(1-\delta)(\epsil/\mu)^{1/2}\Ehp\vE
\Big)\di t \\
&+\int_\FR\Big((1-\delta)\Ehp \vH+\delta(\mu/\epsil)^{1/2}\Hhp\vH 
+\delta\Hhp\vE+(1-\delta)(\epsil/\mu)^{1/2}\Ehp\vE
\Big)\di t.
\end{align*}
Similarly, the terms on lateral sides of the linear functional
$\ell_{TDG}(\cdot)$ become
\begin{align*}
\ell_{TDG}^\calR (\vE,\vH)
=\ldots\;&+
\int_\FL\Big(\delta\epsil^{-1/2} \vH+(1-\delta)\mu^{-1/2}\vE\Big)\,g_L\di t \\
&+\int_\FR\Big(-\delta\epsil^{-1/2}\vH+(1-\delta)\mu^{-1/2}\vE\Big)\,g_R\di t;
\end{align*}
the same holds for $\ell_{DG}(\cdot)$.
We also modify the terms on $\FL\cup\FR$ in the DG norm in \eqref{eq:DGNorms} as follows:
\begin{align*}
\Tnorm{(\vE,\vH)}^2_{DG^\calR}=\ldots+
\N{(1-\delta)^{1/2}(\epsil/\mu)^{1/4}\vE}_{L^2(\FL\cup\FR)}^2
+\N{\delta^{1/2}(\mu/\epsil)^{1/4}\vH}_{L^2(\FL\cup\FR)}^2.
\end{align*}
Note that now, on $\FL$ and $\FR$, both $\vE$ and $\vH$ are
controlled by the DG${}^\calR$ norm.
For this reason, in view of establishing a continuity property for  $a_{TDG}^\calR(\cdot,\cdot)$, we 
define the DG${}^{\calR+}$ norm
to be equal to the DG${}^+$ norm in \eqref{eq:DGNorms}
after removing the last term on the lateral sides.

The coercivity property of Lemma \ref{lem:DGisNorm} and
Proposition~\ref{prop:Coercivity} holds without modifications for $a_{TDG}^\calR(\cdot,\cdot)$ and $\Tnorm{\cdot}_{DG^\calR}$.
The continuity constant of the sesquilinear form now depends on the
parameter $\delta$:
\begin{align*}
\abs{a_{TDG}^\calR (E,H;\vE,\vH)}
\le C_c^\calR\,
\Tnorm{(E,H)}_{DG^{\calR+}}\, \Tnorm{(\vE,\vH)}_{DG^\calR} 
\end{align*}
for all $(E,H),(\vE,\vH)\in \bT(\calT_h)$
with
\begin{align}\label{eq:RobinCc}
C_c^\calR:=\sqrt2\bigg(1+\max\bigg\{\frac{1-\delta_*}{\delta_*};\frac{\delta^*}{1-\delta^*}\bigg\}\bigg) ^{1/2}.
\end{align}
Note that the simplest choice $\delta=1/2$ on $\FL\cup\FR$ gives $C_c^\calR=2\sqrt2$.
As in Theorem~\ref{th:QuasiOpt}, the quasi-optimality follows:
\begin{align}\label{eq:RobinQuasiOpt}
\Tnorm{(E,H)-(\Ehp,\Hhp)}_{DG^\calR}
\le  (1+C_c^\calR) \inf_{(\vE,\vH)\in\bV_p\Th}
\Tnorm{(E,H)-(\vE,\vH)}_{DG^{\calR+}}.
\end{align}

The linear functional $\ell^\calR_{TDG}(\cdot)$ is now bounded in the DG${}^\calR$
norm, even for non-zero boundary conditions, thus the Trefftz-DG
solution depends continuously on the problem data.
This is a slightly stronger property than that holding in the Dirichlet case, see Remark \ref{rem:ContinuousDep}.

Homogeneous Robin boundary conditions, as written in \eqref{eq:RobinIVP} with $g_L=g_R=0$, correspond to absorbing materials, i.e.\ waves hitting the boundary are not reflected back into the domain~$Q$.
In the non-homogeneous case, $g_L$ and $g_R$ define the wave components entering $Q$ from the sides.
This is reflected in the following energy identity for the continuous problem:
\begin{align*}
\calE^\calR(T)&=\calE^\calR(0)+\int_{\FL}EH\di t-\int_{\FR}EH\di t\\
&=\calE^\calR(0)-\int_{\FL\cup\FR}\big(\delta(\mu/\epsil)^{1/2} H^2
+(1-\delta)(\epsil/\mu)^{1/2} E^2\big)\di t\\
&+\int_{\FL}\Big(\delta\epsil^{-1/2}H+(1-\delta)\mu^{-1/2}E\Big)g_L\di t
+\int_{\FR}\Big(-\delta\epsil^{-1/2}H+(1-\delta)\mu^{-1/2}E\Big)g_R\di t.
\end{align*}
The second equality is derived by splitting
  $EH=\delta EH+(1-\delta)EH$, then substituting in the first and
  second term the expressions of $E$ and $H$, respectively, given by
  the boundary conditions in~\eqref{eq:RobinIVP}. Note that the value of the right-hand
  side is the same for any $\delta\in [0,1]$.
This identity is closely replicated by the Trefftz-DG discretisation: in the evolution \eqref{eq:DiscreteEnergyEvo} of the
discrete energy $\calE_{hp}$, the terms on $\FL$ and $\FR$ are
substituted by
\begin{align*}
&\calE_{hp}^\calR(T)=\ldots -\int_{\FL\cup\FR}\big(\delta(\mu/\epsil)^{1/2}
\Hhp^2+(1-\delta)(\epsil/\mu)^{1/2} \Ehp^2\big)\di t\\
&+\int_{\FL}\!\Big(\delta\epsil^{-1/2}\Hhp+(1-\delta)\mu^{-1/2}\Ehp\Big)g_L\di t
+\int_{\FR}\!\Big(-\delta\epsil^{-1/2}\Hhp+(1-\delta)\mu^{-1/2}\Ehp\Big)g_R\di t.
\end{align*}

Defining $\alpha$ on $\FL\cup\FR$ to be equal to
$(1-\delta)(\epsil/\mu)^{1/2}$, we note that
$\Tnorm{(\wE,\wH)}_{DG}\le \Tnorm{(\wE,\wH)}_{DG^\calR}$ for all
Trefftz functions $(w_E,w_H)\in\bT\Th$.
This guarantees that the result of Proposition~\ref{prop:Duality},
namely the control of the $L^2(Q)$ norm with the DG norm for Trefftz
functions, holds for the DG${}^\calR$ norm as well.

Combining the results sketched in this section with the approximation
bounds derived in section~\ref{s:BestApprox} (which are independent of
the type of boundary conditions employed), we obtain convergence
estimates for the Trefftz-DG scheme for problem \eqref{eq:RobinIVP};
this is addressed in Remarks~\ref{rem:RobinConvRates} and
\ref{rem:RobinAnalytic} below.

\section{Best approximation estimates}\label{s:BestApprox}

\subsection{Left- and right-propagating waves}

In order to approximate the solutions of the Maxwell system, we decompose them into two components, one propagating to the right and one to the left.
This allows to represent the solutions in terms of two functions of one real variable.
In this section we describe the relation between fields defined in the space--time domain and their one-dimensional representations.
In the next section this will be used to reduce the proof of approximation estimates for Trefftz spaces to classical one-dimensional polynomial approximation results.

Let $D=(x_0,x_1)\times(t_0,t_1)$ be a space--time rectangle such that $\epsil$ and $\mu$ are constant in it. 
In correspondence to $D$, we define the two intervals
\begin{equation}\label{eq:Omegapm}
\begin{aligned}
\Omega_{D}^- &:=(x_0-ct_1,x_1-ct_0),\\
\Omega_{D}^+ &:=(x_0+ct_0,x_1+ct_1),
\end{aligned}
\end{equation}
and denote their length by
\begin{equation}\label{eq:hD}
 h_D:= x_1-x_0+c(t_1-t_0).
\end{equation}
Their relevance is the following: the restriction to $D$ of the solution of a Maxwell initial value problem posed in $\IR\times\IR^+$ will depend only on the initial conditions posed on $\ODm \cup\ODp $; see Figure~\ref{fig:ODpm}.

\begin{figure}[htb!]  \centering
\begin{tikzpicture}[scale=1]
\draw [->] (-1.3,-1.2)--(-1.3,2);\draw(-1.1,2)node{$t$};
\draw [very thick] (0,0)--(1.5,0)--(1.5,1)--(0,1)--(0,0);\draw(0.75,0.5)node{$D$};
\draw [->] (-4.5,-1) -- (6,-1);\draw(6.2,-1)node{$x$}; 
 \draw (0,1)--(-4,-1); 
 \draw (1.5,0)--(-0.5,-1); 
 \draw (1.5,1)--(5.5,-1); 
 \draw (0,0)--(2,-1); 
\draw [ultra thick] (-4,-1)--(-0.5,-1);\draw(-2,-1.3)node{$\ODm$}; 
\draw [ultra thick] (2,-1)--(5.5,-1);  \draw(3.85,-1.3)node{$\ODp$}; 
\draw [dashed] (0,0)--(0,-1);		\draw(0,-1.3)node{$x_0$};
\draw [dashed] (1.5,0)--(1.5,-1);	\draw(1.5,-1.3)node{$x_1$};
\draw [dashed] (0,0)--(-1.3,0);		\draw(-1.6,0)node{$t_0$};
\draw [dashed] (0,1)--(-1.3,1);		\draw(-1.6,1)node{$t_1$};
\end{tikzpicture}
\caption{The intervals $\ODpm$ in \eqref{eq:Omegapm} corresponding to the space--time rectangle $D$.}\label{fig:ODpm}
\end{figure}
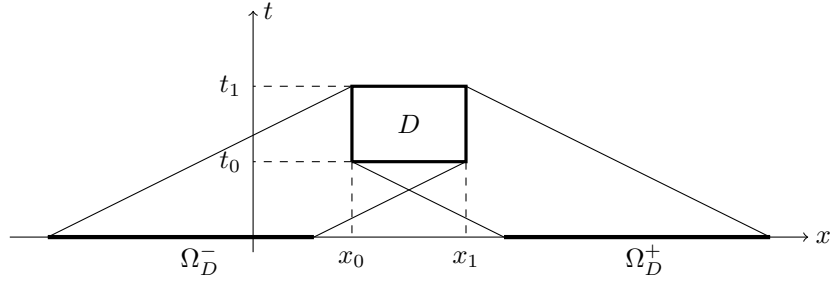

Let $\malpha=(\alpha_x,\alpha_t)\in \IN_0^2$ be a multi-index;
for a sufficiently smooth function $v$, 
we define its
anisotropic derivative $D_c^\malpha v$ as
\[
D_c^\malpha v(x,t):=\frac{1}{c^{\alpha_t}}D^\malpha v(x,t)
=\frac{1}{c^{\alpha_t}}\frac{\partial^{\abs{\malpha}}v(x,t)}{\partial_x^{\alpha_x}\partial_t^{\alpha_t}}.
\]
Note that, if $u$ and $w$ satisfy
\begin{equation}\label{eq:leftright}
u(x,t)=u_0(x-ct),\qquad w(x,t)=w_0(x+ct),
\end{equation}
with $u_0$ and $w_0$ defined in $\ODm $ and $\ODp $, respectively, then
\begin{align*}
D_c^\malpha u(x,t)&=(-1)^{\alpha_t}u_0^{(\abs{\malpha})}(x-ct),\\
D_c^\malpha w(x,t)&=w_0^{(\abs{\malpha})}(x+ct).
\end{align*}

We define the Sobolev spaces $W^{j,\infty}_c(D)$ and $H^j_c(D)$ as the
spaces of functions whose $D^\malpha_c$ derivatives,
$0\le\abs{\malpha}\le j$, belong to $L^\infty(D)$ and $L^2(D)$, respectively.
We define the following seminorms:
\[
\begin{aligned}
&\abs{v}_{W^{j,\infty}_c(D)}:=\sup_{\abs{\malpha}=j}\N{D^\malpha_c v}_{L^\infty(D)}, \qquad
&\abs{v}_{H^j_c(D)}^2 :=\sum_{\abs{\malpha}=j}\N{D^\malpha_c
  v}_{L^2(D)}^2.
\end{aligned}
\]
Note that for $j=0$ they reduce to the usual $L^\infty(D)$ and
$L^2(D)$ norms and we omit the subscript $c$.
On the segments $\ODpm$, 
the $W^{j,\infty}(\ODpm)$ and $H^j(\ODpm)$ seminorms are defined in
the standard way.
We finally define
the weighted $H^1_c(D)$ norm (recall the definition of $h_D$ in~\eqref{eq:hD})
\begin{equation}\label{eq:H1cnorm}
\N{v}^2_{H^1_c(D)}:= h_D^{-1}\N{v}^2_{L^2(D)}+ h_D\abs{v}^2_{H^1_c(D)}.
\end{equation}

\begin{proposition}\label{prop:Pullback}
Assume that $u(x,t)=u_0(x-ct)$ for $(x,t)\in D$. 
Then, for $j\in\IN_0$,
\begin{itemize}
\item[(i)] $u\in W^{j,\infty}_c(D)$ if and
only if $u_0\in W^{j,\infty}(\ODm )$, and
\[
\abs{u}_{W^{j,\infty}_c(D)}=\abs{u_0}_{W^{j,\infty}(\ODm )};
\]
\item[(ii)] if $u\in W^{j,\infty}_c(D)$, then $u_0\in H^j(\ODm )$, and 
\[
\abs{u_0}_{H^j(\ODm )}^2\le
h_D\, 
\abs{u}_{W^{j,\infty}_c(D)}^2;
\]
\item[(iii)]
if $u_0\in H^j(\ODm )$, then $u\in H^j_c(D)$, and
\begin{align*}
\abs{u}_{H^j_c(D)}^2&\le
(j+1)\min\Big\{(t_1-t_0),\frac{(x_1-x_0)}{c} \Big\}
\abs{u_0}_{H^j(\ODm )}^2.
\end{align*}
Furthermore, if $j=1$,
$$ \N{u}_{H^1_c(D)}^2\le
\frac1c\N{u_0}_{L^2(\ODm )}^2+  \frac{2h_D^2}c\abs{u_0}_{H^1(\ODm )}^2.
$$
\end{itemize}
A similar result holds for $w(x,t)=w_0(x+ct)$, with $\ODp $ instead of $\ODm $.
\end{proposition}

\begin{proof}
For the $W^{j,\infty}_c(D)$-seminorms in \textit{(i)}, we have
\[
\abs{u}_{W^{j,\infty}_c(D)}=\sup_{\abs{\malpha}=j}\N{D^\malpha_c
  u}_{L^\infty(D)}
=\sup_{\abs{\malpha}=j}\N{u_0^{(\abs{\malpha})} (x-ct)}_{L^\infty(D)}=\abs{u_0}_{W^{j,\infty}(\ODm )}.
\] 

For the bound of $\abs{u_0}_{H^j(\ODm )}^2$ in \textit{(ii)}, we have
\[
\begin{aligned}
\abs{u_0}_{H^j(\ODm )}^2&=\int_{\ODm }\abs{u_0^{(j)}(z)}^2\di z
\le\abs{\ODm }\sup_{z\in\ODm }\abs{u_0^{(j)}(z)}^2
=\abs{\ODm }\sup_{\abs{\malpha}=j}\sup_{(x,t)\in D}
\abs{D^\malpha_c u(x,t)}^2\\
&=h_D 
\abs{u}_{W^{j,\infty}_c(D)}^2.
\end{aligned}
\]

Consider now the $H^j_c(D)$-seminorm. We have
\begin{align*}
\abs{u}_{H^j_c(D)}^2&=\sum_{\abs{\malpha}=j}\iint_D\abs{D^\malpha_c u}^2
\di x\di t\\
&=\sum_{\abs{\malpha}=j}\iint_D\abs{u_0^{(\abs{\malpha})}(x-ct)}^2\di x\di t\\
&=(j+1)\iint_D\abs{u_0^{(j)}(x-ct)}^2\di x\di t\\
& = 
(j+1)\min\Big\{
\int_{t_0}^{t_1}\N{u_0^{(j)}}_{L^2(x_0-ct,x_1-ct)}^2\di t,
\int_{x_0}^{x_1}\frac{1}{c}\N{u_0^{(j)}}_{L^2(x-ct_1,x-ct_0)}^2\di  x
\Big\}\\
&\le (j+1)\min\Big\{
\int_{t_0}^{t_1}\N{u_0^{(j)}}_{L^2(\ODm )}^2\di t,
\int_{x_0}^{x_1}\frac{1}{c}\N{u_0^{(j)}}_{L^2(\ODm )}^2\di x
\Big\}\\
&\le (j+1)\min\Big\{(t_1-t_0),
\frac{(x_1-x_0)}{c} \Big\}\N{u_0^{(j)}}_{L^2(\ODm )}^2,
\end{align*}
from which the desired bound in \textit{(iii)} follows. Note that the
two terms in the curly braces in the last equality are equal to each other.

The result for $w$ is obtained in the same way.
\end{proof}

\begin{rem}\label{rem:ConversePullback}
The inequality opposite to that in
item \textit{(iii)} of 
Proposition~\ref{prop:Pullback} is not true.
For example, the functions $u_{0,\ell}(z)=\sqrt{\ell}\,\chi_{(x_0-ct_1, x_0-ct_1+\ell^{-1})}(z)$ (where $\chi_{(a,b)}$ denotes the characteristic function of the interval $(a,b)$) belong to $L^2(\ODm)$ for sufficiently large $\ell\in\IN$, and $\N{u_{0,\ell}}_{L^2(\ODm)}=1$.
On the other hand, $u_\ell(x,t)=u_{0,\ell}(x-ct)\in L^2(D)$ but $\N{u_\ell}_{L^2(D)}=(2c\ell)^{-1/2}$, so no bound 
$\N{u_{0,\ell}}_{L^2(\ODm)}\le C\N{u_\ell}_{L^2(D)}$ with $C$
independent of $\ell$ is possible.
\end{rem}

\subsection{Local discrete Trefftz spaces}\label{ss:LocalTrefftz}
Given a rectangle $D$ as above, we define the corresponding local Trefftz space as
$$\bT(D):=\bigg\{(E ,H )\in H^1(D)^2,\quad
\der{E }x+\der{(\mu H )}t=\der{H }x+\der{(\epsil E )}t=0 \quad
\text{in}\;  D \bigg\}.$$
Any Trefftz field $(E,H)\in\bT(D)$ can be decomposed as
\begin{equation}\label{eq:EHuw}
(E,H)=\left(\frac{u+w}{2 \epsil^{1/2}},\,\frac{u-w}{2 \mu^{1/2}}\right),\quad
\text{with}\quad
u=\epsil^{1/2} E+\mu^{1/2} H \quad\text{and}\quad
w=\epsil^{1/2} E-\mu^{1/2} H.
\end{equation}
The waves $u$ and $w$ satisfy~\eqref{eq:leftright}, i.e.\ they propagate in the right and the left direction, respectively.

Conversely, for any $u_0\in C^m({\overline\Omega_D^-})$ and $w_0\in
C^m({\overline\Omega_D^+})$ for $m\in\IN_0$, the functions $u$ and $w$ as defined in \eqref{eq:leftright} satisfy the wave equation~\eqref{eq:WaveEq} in $D$, 
and the field $(E,H)$ obtained combining them as in \eqref{eq:EHuw} belongs to 
$\bT(D)\cap C^m(\overline D)^2$.

This suggests a construction of discrete subspaces of $\bT(D)$:
given $p\in\IN_0$ and two sets of $p+1$ 
linearly independent functions
${\Phi^- =}\{\varphi_0^-,\ldots,\varphi_p^-\}\subset C^m({\overline\Omega_D^-})$
and  ${\Phi^+ =}\{\varphi_0^+,\ldots,\varphi_p^+\}\subset
C^m({\overline\Omega_D^+})$ we define the space
\begin{align*}
\bV_p(D):=\spn\bigg\{
&\Big(\frac{\varphi_0^-(x-ct)}{2\epsil^{1/2}},\,  \frac{\varphi_0^-(x-ct)}{2\mu^{1/2}}\Big),\ldots,
 \Big(\frac{\varphi_p^-(x-ct)}{2\epsil^{1/2}},\,  \frac{\varphi_p^-(x-ct)}{2\mu^{1/2}}\Big),\\
&\Big(\frac{\varphi_0^+(x+ct)}{2\epsil^{1/2}},\, -\frac{\varphi_0^+(x+ct)}{2\mu^{1/2}}\Big),\ldots,
 \Big(\frac{\varphi_p^+(x+ct)}{2\epsil^{1/2}},\, -\frac{\varphi_p^+(x+ct)}{2\mu^{1/2}}\Big)
\bigg\}, 
\end{align*}
which is a subspace of $\bT(D)\cap C^m(\overline D)^2$ with dimension $2(p+1)$.

By virtue of Proposition~\ref{prop:Pullback}, the approximation
properties of $\bV_p(D)$ in $\bT(D)$ only depend on the approximation
properties of the one-dimensional functions
$\{\varphi_0^\pm,\ldots,\varphi_p^\pm\}$:
for all $(E,H)\in\bT(D)\cap W^{j,\infty}_c(D)^2$, 
defining $u$, $w$, $u_0$ and $w_0$ from $(E,H)$ using \eqref{eq:EHuw} and \eqref{eq:leftright},
\begin{align}\label{eq:infEHuw}
&\inf_{(\Ehp,\Hhp)\in\bV_p(D)}
\Big(\abs{\epsil^{1/2}(E-\Ehp)}_{W^{j,\infty}_c(D)}
+\abs{\mu^{1/2}(H-\Hhp)}_{W^{j,\infty}_c(D)}\Big)\\
&\qquad\overset{\eqref{eq:EHuw}}=\inf_{\substack{u_{0,p}\in\spn\{\varphi_0^-,\ldots,\varphi_p^-\},
\nonumber\\
w_{0,p}\in\spn\{\varphi_0^+,\ldots,\varphi_p^+\} }}
\bigg(\frac{1}{2}\abs{{u(x,t)-u_{0,p}(x-ct)+w(x,t)-w_{0,p}(x+ct)}}_{W^{j,\infty}_c(D)}
\nonumber\\
&\hspace{45mm}+\frac{1}{2}\abs{{u(x,t)-u_{0,p}(x-ct)-w(x,t)+w_{0,p}(x+ct)}}_{W^{j,\infty}_c(D)}
\bigg)\nonumber\\
&\overset{\text{Prop.~\ref{prop:Pullback} \textit{(i)}}}\le
\inf_{u_{0,p}\in\spn\{\varphi_0^-,\ldots,\varphi_p^-\}}\abs{u_0-u_{0,p}}_{W^{j,\infty}(\ODm)}
+\inf_{w_{0,p}\in\spn\{\varphi_0^+,\ldots,\varphi_p^+\}}\abs{w_0-w_{0,p}}_{W^{j,\infty}(\ODp )},
\nonumber
\end{align}
while for all $(E,H)\in\bT(D)\cap H^j_c(D)^2$
\begin{align}\label{eq:infEHuwHj}
&\inf_{(\Ehp,\Hhp)\in\bV_p(D)}
\Big(\abs{\epsil^{1/2}(E-\Ehp)}_{H^j_c(D)}^2
+\abs{\mu^{1/2}(H-\Hhp)}_{H^j_c(D)}^2\Big)\\
&\overset{\text{Prop.~\ref{prop:Pullback}
    \textit{(iii)}}}\le(j+1)\min\Big\{(t_1-t_0),\frac{(x_1-x_0)}{c}
\Big\}
\nonumber\\
&\hspace{15mm}
\bigg(\inf_{u_{0,p}\in\spn\{\varphi_0^-,\ldots,\varphi_p^-\}}\abs{u_0-u_{0,p}}_{H^j(\ODm)}^2
+\inf_{w_{0,p}\in\spn\{\varphi_0^+,\ldots,\varphi_p^+\}}\abs{w_0-w_{0,p}}_{H^j(\ODp )}^2\bigg).
\nonumber
\end{align}

In the following, we are going to consider polynomial bases;
  alternative choices, e.g.\ trigonometric functions, are
possible, as suggested in \cite[section~3.1]{PFT09}.

\subsection{Polynomial Trefftz spaces}\label{ss:PolyTrefftz}

The most straightforward choice for the space $\bV_p(D)$ is to take a
polynomial basis: $\Phi^-=\Phi^+=\{z^j\}_{j=0}^p$.
(Of course, in practical implementations of the Trefftz-DG method different choices of the basis for the same space might be preferred, e.g.\ Legendre polynomials; this however does not affect the approximation properties of the discrete space and the orders of convergence of the scheme.)
In this case, the general 
field $(\vE,\vH)\in\bV_p(D)$ can be written as
\begin{align}\label{eq:localpolyn}
v_E(x,t)&={\epsil^{-1/2} a_0+} \epsil^{-1/2}\sum_{{j=1}}^{p} a_j({x-ct})^j +
\epsil^{-1/2}\sum_{j=1}^{p} b_j({x+ct})^j, \\
v_H(x,t)&={\mu^{-1/2} b_0+}\mu^{-1/2}\sum_{{j=1}}^{p} a_j({x-ct})^j -
\mu^{-1/2}\sum_{j=1}^{p} b_j({x+ct})^j, \quad a_j,b_j\in\IR,\ (x,t)\in D,
\nonumber
\end{align}
Note that the space $\bV_p(D)$ has dimension {$2p+2$}, while the full
polynomial space of the same degree $(\IP^p)^2$ has  
much higher
dimension $(p+1)(p+2)$ but similar approximation properties for solutions of wave equations, as
  demonstrated for example in Figure~\ref{fig:ConvpVersion} below.
Instead of $(\IP^p)^2$, tensor product polynomial spaces could be considered; the space dimension would be $2(p+1)^2$ and the approximation rates would remain unchanged.

We recall that our goal is to control the best approximation error in
the DG${}^+$~norm. 
Observing its definition \eqref{eq:DGNorms}, it is enough to control
the error either in $L^\infty(Q)$ or in $H^1_c(Q)$.
Following the second 
route, 
we prove simple approximation bounds in $H^1_c(Q)$, 
for a general rectangle $D\subset Q$, which will be
chosen as a mesh element in section~\ref{s:ConvRates}.

We first consider approximation bounds for functions
with limited Sobolev regularity; then, we prove approximation estimates with exponential rates in the polynomial degree for analytic functions.

\subsubsection{Algebraic approximation}\label{sss:AlgebraicApprox}

Classical $hp$-approximation results, together with a scaling argument, give the following one-dimensional polynomial approximation bound (see \cite[Corollary~3.15]{Schwab98}, with the norms defined in \cite[equation (3.3.10)]{Schwab98}): for all $u\in H^{s+1}(a,b)$, $s,p\in\IN$, and $s\le p$,
\begin{align}\label{eq:SchwabApprox}
\inf_{P\in\IP^p([a,b])} 
\bigg( \frac{p(p+1)}{(b-a)^2}\N{u-P}^2_{L^{2}(a,b)} +\abs{u-P}^2_{H^1(a,b)}\bigg)
&\le 2\frac{(p-s)!}{(p+s)!}
(b-a)^{2s}\abs{u}^2_{H^{s+1}(a,b)},
\end{align}
where $\IP^{p}([a,b])$ denotes the space of polynomials of degree at most $p$ in the interval $[a,b]$.
Combining \eqref{eq:SchwabApprox} with the bound \eqref{eq:infEHuwHj} proved in the
previous section, the definitions of $\ODpm$
\eqref{eq:Omegapm}, $h_D$ \eqref{eq:hD},
$u,w$ \eqref{eq:EHuw}, and $u_0,w_0$ \eqref{eq:leftright}, we have for all
$(E,H)\in \bT\Th\cap W^{s+1,\infty}(D)^2$
\begin{align}\label{eq:FinalBestApproxNEW}
\inf_{(\Ehp,\Hhp)\in\bV_p(D)}&
\Big(\N{\epsil^{1/2}(E-\Ehp)}_{H^1_c(D)}^2
+\N{\mu^{1/2}(H-\Hhp)}_{H^1_c(D)}^2\Big)\\
&\le
\frac 4c \; \frac{(p-s)!}{(p+s)!}\; h_D^{2s+2}
\Big(\abs{u_0}_{H^{s+1}(\ODm)}^2+\abs{w_0}_{H^{s+1}(\ODp)}^2\Big)
\nonumber\\
&\le
\frac 4c \; \frac{(p-s)!}{(p+s)!}\; h_D^{2s+3}
\Big(\abs{u_0}_{W^{s+1,\infty}(\ODm)}^2+\abs{w_0}_{W^{s+1,\infty}(\ODp)}^2\Big)
\nonumber\\
&\hspace{-5mm}\overset{\text{Prop.~\ref{prop:Pullback}\textit{(i)}}}=
\frac 4c \; \frac{(p-s)!}{(p+s)!}\; h_D^{2s+3}
\Big(\abs{u}_{W^{s+1,\infty}_c(D)}^2+\abs{w}_{W^{s+1,\infty}_c(D)}^2\Big)
\nonumber\\
&\overset{\eqref{eq:EHuw}}\le
\frac{16}c \; \frac{(p-s)!}{(p+s)!}\; h_D^{2s+3}
\Big(\abs{\epsil^{1/2} E}_{W^{s+1,\infty}_c(D)}^2+\abs{\mu^{1/2} H}_{W^{s+1,\infty}_c(D)}^2\Big).
\nonumber
\end{align}
%
This bound can be used as best approximation estimate in the
convergence analysis of the Trefftz-DG method; we will do this in
section~\ref{s:ConvRates}.
Combining a suitable variant of \eqref{eq:SchwabApprox} with the bounds
in section~\ref{ss:LocalTrefftz}, one could easily obtain similar bounds
in $W^{j,\infty}_c(D)$ and $H^j_c(D)$; however, since the
inequality opposite to that
of item \emph{(iii)} in Proposition~\ref{prop:Pullback} does not hold
(see Remark~\ref{rem:ConversePullback}), we are not able to obtain $H^j_c(D)$
norms  of $(E,H)$ at the right-hand side of the approximation bounds.

\subsubsection{Exponential approximation}\label{sss:ExponentialApprox}

The degree $p$ of the polynomial Trefftz discrete space enters the approximation bound \eqref{eq:FinalBestApproxNEW} through the factor $(p-s)!/(p+s)!$, which leads to algebraic convergence with order depending only on the solution regularity $s$.
Classical polynomial approximation theory shows that, if $E$ and $H$
admit analytic extension in a complex neighbourhood of $D$, then
exponential convergence in the polynomial degree $p$ is achieved.

Indeed, if $u_0$ and $w_0$ are analytic in the complex ellipses with foci
at the extrema of $\ODm$ and $\ODp$, respectively, and sum of the
semiaxes equal to $\rho_D(x_1-x_0+c(t_1-t_0))/2$ for $\rho_D>1$,
then by the classical Bernstein theorem (e.g.\ \cite[Chapter~7,
Theorem~8.1]{DeVoreLorentz}) the exponential convergence rate 
in $L^\infty(\ODpm)$
of the polynomial
approximation of $u_0$ and $w_0$ is $\rho_D$:
\begin{equation}\label{eq:Bernstein}
\inf_{P\in\IP^p(\ODm)}\N{u_0-P}_{L^\infty(\ODm)}+
\inf_{P\in\IP^p(\ODp)}\N{w_0-P}_{L^\infty(\ODp)} \le
C_{D,\mathrm{Bern}}\, \rho_D^{-p}
\end{equation}
for some constant $C_{D,\mathrm{Bern}}>0$ independent of $p$.
As in~\eqref{eq:FinalBestApproxNEW}, combining \eqref{eq:infEHuw} and~\eqref{eq:Bernstein}, we
obtain the following exponential approximation bound in the space--time
rectangle:
\begin{align}\label{eq:ExpApprox}
\inf_{(\Ehp,\Hhp)\in\bV_p(D)}
\Big(\N{\epsil^{1/2}(E-\Ehp)}_{L^{\infty}(D)}
+\N{\mu^{1/2}(H-\Hhp)}_{L^{\infty}(D)}\Big)
\le C_{D,\mathrm{Bern}}\, \rho_D^{-p}.
\end{align}

\section{Convergence rates}\label{s:ConvRates}

We now derive the convergence rates of the Trefftz-DG method with polynomial approximating spaces
\begin{equation}
\bV_p\Th=\big\{(\vE,\vH)\in L^2(Q)^2:\ (\vE,\vH)_{|_K}\ \text{are as
  in~\eqref{eq:localpolyn} with $p=p_K$}\big\}.
\label{eq:VpTh}
\end{equation}
The two main ingredients are the quasi-optimality results investigated in
section~\ref{s:Analysis} and the best approximation bounds proved in
section~\ref{s:BestApprox}.
To combine them, we need to control the DG${}^+$~norm
\eqref{eq:DGNorms} of the approximation error with its
$H^1_c(Q)$~norm, weighted with $\epsil$ and $\mu$,
to be able to use the bound \eqref{eq:FinalBestApproxNEW}.
To this purpose, we define the following parameters:
\begin{align}\label{eq:zeta}
\zeta_K:=\max\bigg\{
&\N{\alpha\epsil^{-1}}_{L^\infty(\partial K\cap(\Fver\cup\FL\cup\FR))};
&&\N{\alpha^{-1}\mu^{-1}}_{L^\infty(\partial K\cap(\Fver\cup\FL\cup\FR))};\\
&
\N{\beta\mu^{-1}}_{L^\infty(\partial K\cap\Fver)};\quad
&&\N{\beta^{-1}\epsil^{-1}}_{L^\infty(\partial K\cap\Fver)}
\bigg\}
&&\forall K\in\calT_h.
\nonumber
\end{align}
For any mesh element $K\in\calT_h$, we denote by $h_K^x$, $h_K^t$
the lengths of its horizontal and vertical edges, respectively, i.e.\
the local meshwidths of the discretisation.

Before stating our main convergence theorem, we prove a simple
explicit trace inequality for functions defined on
rectangles.
\begin{lemma}
Given a space--time rectangle $D=(x_0,x_1)\times(t_0,t_1)$, denote by 
$\partial D^{\mathrm{SN}}=(x_0,x_1)\times\{t_0,t_1\} $ 
and 
$\partial D^{\mathrm{WE}}= \{x_0,x_1\}\times(t_0,t_1)$ 
the decomposition of its boundary in opposite sides.
For all $u\in H^1(D)$, we have the following trace estimates:
\begin{align}
\label{eq:Trace}
\N{u}_{L^2(\partial D^{\mathrm{SN}})}^2&\le
\frac{4}{t_1-t_0} \N{u}_{L^2(D)}^2  + \frac{t_1-t_0}2\N{\frac{\partial u}{\partial t}}_{L^2(D)}^2,\\
\N{u}_{L^2(\partial D^{\mathrm{WE}})}^2&\le
\frac{4}{x_1-x_0} \N{u}_{L^2(D)}^2  + \frac{x_1-x_0}2\N{\frac{\partial u}{\partial x}}_{L^2(D)}^2.
\nonumber
\end{align}
\end{lemma}
\begin{proof}
We assume without loss of generality that $D$ is centred at the origin, i.e.\ $x_0=-x_1$ and $t_0=-t_1$.
The first bound is a simple consequence of the fundamental theorem of calculus:
\begin{align*}
\N{u}_{L^2(\partial D^{\mathrm{SN}})}^2&=
\int_{\{-t_1,t_1\}\times (-x_1,x_1)} u^2\di x \\
&=\frac1{t_1}\iint_D \frac{\partial (t u^2)}{\partial t}\di x\di t\\
&=\frac1{t_1}\iint_D \Big(u^2 + 2t u \frac{\partial u}{\partial t}\Big)\di x\di t\\
&\le \frac1{t_1} \N{u}_{L^2(D)}^2  + 2 \N{u}_{L^2(D)}\N{\frac{\partial u}{\partial t}}_{L^2(D)}
\le \frac2{t_1} \N{u}_{L^2(D)}^2  + t_1\N{\frac{\partial u}{\partial t}}_{L^2(D)}^2.
\end{align*}
The first inequality in \eqref{eq:Trace} follows from $t_1=(t_1-t_0)/2$; the second bound is derived in a similar way.
\end{proof}

\begin{theorem}\label{th:Convergence}
For all $K\in\calT_h$, fix $p_K,s_K\in\IN$ with $1\le s_K\le p_K$ and assume that the restriction to $K$ of the solution $(E,H)$ of the initial boundary value problem~\eqref{eq:IVP} with $J=0$ belongs to 
$W^{s_k+1,\infty}(K)$. 
Define the discrete space $\bV_p\Th\subset\bT\Th$
as in~\eqref{eq:VpTh},
and let $(\Ehp,\Hhp)$ be the solution of the corresponding Trefftz-DG variational formulation \eqref{eq:VFTDG}.
Then, the following bound holds true:
\begin{align}\nonumber
&\Tnorm{(E,H)-(\Ehp,\Hhp)}_{DG}\\
&\le
 \frac{12}{\sqrt c}\sum_{K\in\calT_h}
\bigg(6\Big(c+\frac{h_K^x}{h_K^t}\Big)+8\zeta_K \Big(1+c \frac{h_K^t}{h_K^x}\Big) \bigg)^{1/2}
{(e/2)^{\frac{s_K^2}{p_K}} \;\frac{\big(h_K^x+ch_K^t\big)^{s_K+\frac32}}{p_K^{s_K}}}
\nonumber\\
&\hspace{30mm}\cdot
\Big(\abs{\epsil^{1/2} E}_{W^{s_K+1,\infty}_c(K)}+\abs{\mu^{1/2} H}_{W^{s_K+1,\infty}_c(K)}\Big).
\label{eq:DGConvergence}
\end{align}

If the bound \eqref{eq:DualStability} holds true for the solution of the auxiliary problem~\eqref{eq:dualIVP}, we also have the following bound in $L^2(Q)$:
\begin{align}\nonumber
&\Big(\N{\mu^{-1/2}(E-\Ehp)}_{L^2(Q)}^2
+\N{\epsil^{-1/2}(H-\Hhp)}_{L^2(Q)}^2\Big)^{1/2}
\\
&\le 
 \frac{12\sqrt{2}}{\sqrt c}\Cstab\sum_{K\in\calT_h}
\bigg(6\Big(c+\frac{h_K^x}{h_K^t}\Big)+8\zeta_K \Big(1+c \frac{h_K^t}{h_K^x}\Big) \bigg)^{1/2}
{(e/2)^{\frac{s_K^2}{p_K}} \;\frac{\big(h_K^x+ch_K^t\big)^{s_K+\frac32}}{p_K^{s_K}}}
\nonumber\\
&\hspace{30mm}\cdot
\Big(\abs{\epsil^{1/2} E}_{W^{s_K+1,\infty}_c(K)}+\abs{\mu^{1/2} H}_{W^{s_K+1,\infty}_c(K)}\Big).
\label{eq:L2Convergence}
\end{align}
with $\Cstab$ from \eqref{eq:DualStability}. 
\end{theorem}

\begin{proof}
Given an element $K\in\calT_h$, we denote by $\partial K^{\mathrm N}$, $\partial
K^{\mathrm S}$, $\partial K^{\mathrm W}$, and $\partial K^{\mathrm E}$ its North, South, West and
East sides, respectively, North pointing in the positive time direction, 
 and set $\partial K^{\mathrm{WE}}:=\partial K^{\mathrm W}\cup \partial K^{\mathrm E}$.

For all $(\vE,\vH)\in H^1\Th^2$, expanding the DG${}^+$ norm
\eqref{eq:DGNorms}, using the definition \eqref{eq:zeta} of $\zeta_K$ and the trace inequalities \eqref{eq:Trace}, we have the bound
\begin{align*}
\Tnorm{(\vE,\vH)}_{DG^+}^2&\le\sum_{K\in\calT_h}\bigg(
\frac{1}{2}\N{\epsil^{1/2}\vE}_{L^2(\partial K^{\mathrm S})}^2
+\frac{1}{2}\N{\mu^{1/2}\vH}_{L^2(\partial K^{\mathrm S})}^2\\
&\hspace{15mm}
+\frac{3}{2}\N{\epsil^{1/2}\vE}_{L^2(\partial K^{\mathrm N})}^2
+\frac{3}{2}\N{\mu^{1/2}\vH}_{L^2(\partial K^{\mathrm N})}^2\\
&\hspace{15mm}+\N{\alpha^{1/2}\vE}_{L^2(\partial K^{\mathrm{WE}})}^2
+\N{\beta^{1/2}\vH}_{L^2(\partial
  K^{\mathrm{WE}}\setminus(\calF_h^{L}\cup\calF_h^R))}^2\\
&\hspace{15mm}+\N{\beta^{-1/2}\vE}_{L^2(\partial
  K^{\mathrm{WE}}\setminus(\calF_h^{L}\cup\calF_h^R))}^2
+\N{\alpha^{-1/2}\vH}_{L^2(\partial K^{\mathrm{WE}})}^2
\bigg) \\
&\hspace{-10mm}\overset{\eqref{eq:Trace}}\le\sum_{K\in\calT_h}
\bigg(\frac{6}{h_K^t}+\frac{8\zeta_K}{h_K^x}\bigg)
\Big(\N{\epsil^{1/2}\vE}^2_{L^2(K)}+\N{\mu^{1/2}\vH}^2_{L^2(K)}\Big)\\
&\qquad
+\frac{3h_K^t}4\bigg(\N{\derp{(\epsil^{1/2}\vE)}t}^2_{L^2(K)}
+\N{\derp{(\mu^{1/2}\vH)}t}^2_{L^2(K)}\bigg)\\
&\qquad
+\zeta_K h_K^x\bigg(\N{\derp{(\epsil^{1/2}\vE)}x}^2_{L^2(K)}
+\N{\derp{(\mu^{1/2}\vH)}x}^2_{L^2(K)}\bigg)
\\
&\hspace{-10mm}\overset{\eqref{eq:H1cnorm}}\le\sum_{K\in\calT_h}
\bigg(6\Big(c+\frac{h_K^x}{h_K^t}\Big)+8\zeta_K \Big(1+c \frac{h_K^t}{h_K^x}\Big) \bigg)
\bigg(\N{\epsil^{1/2}\vE}^2_{H^1_c(K)}+\N{\mu^{1/2}\vH}^2_{H^1_c(K)}\bigg).
\end{align*}
Combining this with the quasi-optimality and the approximation results gives the error bound:
\begin{align*}
&\Tnorm{(E,H)-(\Ehp,\Hhp)}_{DG} 
\overset{\eqref{eq:QuasiOpt}}\le  3 \inf_{(\vE,\vH)\in\bV_p\Th}
\Tnorm{(E,H)-(\vE,\vH)}_{DG^+}\\
&\:\le 3\sum_{K\in\calT_h}
\bigg(6\Big(c+\frac{h_K^x}{h_K^t}\Big)+8\zeta_K \Big(1+c \frac{h_K^t}{h_K^x}\Big) \bigg)^{1/2}
\\
&\hspace{30mm} \cdot\inf_{\substack{(\vE,\vH)\\\in\bV_p(K)}}
\Big(\N{\epsil^{1/2}(E-\vE)}_{H^1_c(K)}^2+\N{\mu^{1/2}(H-\vH)}_{H^1_c(K)}^2\Big)^{1/2}
\\
&\overset{\eqref{eq:FinalBestApproxNEW}}\le \frac{12}{\sqrt c}\sum_{K\in\calT_h}
\bigg(\bigg(6\Big(c+\frac{h_K^x}{h_K^t}\Big)+8\zeta_K \Big(1+c \frac{h_K^t}{h_K^x}\Big) \bigg)^{1/2}
\bigg(\frac{(p_K-s_K)!}{(p_K+s_K)!}\bigg)^{1/2}\big(h_K^x+ch_K^t\big)^{s_K+\frac32}\\
&\hspace{30mm}
\cdot\Big(\abs{\epsil^{1/2} E}_{W^{s_K+1,\infty}_c(K)}
+\abs{\mu^{1/2} H}_{W^{s_K+1,\infty}_c(K)}\Big).
\end{align*}
The bound \eqref{eq:DGConvergence} in the assertion follows by applying to the factorial terms the upper and lower Stirling inequalities in the form of \cite[Corollary~3.3]{Batir08}: for all $s\le p\in \IN$
\begin{align*}
\frac{(p-s)!}{(p+s)!}
&\le\underbrace{\Big(\frac{p-s+e^2/2\pi-1}{p+s+1/6}\Big)^{1/2}}_{\le1, \; (s\ge1)}
\frac{(p-s)^{p-s}}{(p+s)^{p+s}} e^{2s}
\le\bigg(\frac{(1-s/p)^{\frac ps-1}}{(1+s/p)^{\frac ps+1}}\;\frac{e^2}{p^2}\bigg)^{s}
\le \frac{(e/2)^{2s^2/p}}{p^{2s}},
\end{align*}
where the last inequality follows noting that $f(x)=(1-x)^{\frac1x-1}(1+x)^{-\frac1x-1}4^xe^{2-2x}\le1$ for all $0<x<1$
(which in turn can be verified by checking the convexity of $\log f$ and its limit values for $x\to0$ and $1$).

The estimate in $L^2(Q)$ norm follows from combining this bound with Corollary~\ref{cor:L2}.
\end{proof}


\begin{rem}
The constant in the first brackets at the right-hand side of the bound \eqref{eq:DGConvergence} controls the loss of accuracy due to anisotropic elements.
If all mesh elements $K$ satisfy $h_x=ch_t$, then the constants reduces to $2(3c+4\zeta_K)^{1/2}$.
\end{rem}

{From bound \eqref{eq:L2Convergence} we see that, if we fix a constant polynomial degree $p_K=p$, consider a sequence of quasi-uniform meshes with $h_x\sim ch_t$ and discretise an initial boundary value problem with constant coefficients and whose solution is sufficiently smooth, then the $L^2$ norm of the Trefftz-DG error converges in the meshwidth with algebraic order equal to $p+1$ (note that in this case $\Cstab\sim h_x^{-1/2}$ in \eqref{eq:Cdual}).
Theorem~5.2 of \cite{MoRi05} (see also Remark~5.2 therein) gives the slightly lower order of convergence ${p+1/2}$ for full polynomial spaces (thus with higher dimension at the same polynomial degree) in a much more general context (higher space dimensions, more general hyperbolic systems, non tensor-product meshes).
The numerical experiments in \cite[section~7]{MoRi05} (confirmed by
those in section~\ref{s:numer-exper} below) recover the higher order
$p+1$.}

The estimates of Theorem~\ref{th:Convergence} suggest to refine the space--time mesh and
reduce the local polynomial degree in the elements where the solution
has lower regularity.
The mesh refinement might be determined a priori knowing the locations
of possible singularities in (the derivatives of) the initial data,
discontinuities of the material parameters and non-matching of (the
derivatives of) the initial and boundary conditions, and
propagating them into $Q$ along the characteristics.


\begin{rem}\label{rem:RobinConvRates}
For the Robin initial boundary value problem~\eqref{eq:RobinIVP}
described in section~\ref{ss:Robin}, an analogous convergence result to
Theorem~\ref{th:Convergence} holds.
After substituting the DG${}^\calR$ norm in place of the DG norm,
the bounds \eqref{eq:DGConvergence} and \eqref{eq:L2Convergence} hold
with a factor $(1+C_c^\calR)/3$ multiplied to the right-hand side, with $C_c^\calR$ as in \eqref{eq:RobinCc} (due to the different quasi-optimality bounds in \eqref{eq:QuasiOpt} and \eqref{eq:RobinQuasiOpt}), and with $\zeta_K$ substituted by
\begin{align*}
\zeta_K^\calR:=\max\bigg\{
&\N{\alpha\epsil^{-1}}_{L^\infty(\partial K\cap\Fver)}; &&
\N{\alpha^{-1}\mu^{-1}}_{L^\infty(\partial K\cap\Fver)};\\
&
\N{\beta\mu^{-1}}_{L^\infty(\partial K\cap\Fver)};&&
\N{\beta^{-1}\epsil^{-1}}_{L^\infty(\partial K\cap\Fver)};\\
&
\N{(1-\delta)(\epsil\mu)^{-1/2}}_{L^\infty(\partial K\cap(\FL\cup\FR))};&&
\N{\delta(\epsil\mu)^{-1/2}}_{L^\infty(\partial K\cap(\FL\cup\FR))}
\bigg\}.
\end{align*}
\end{rem}

\subsection{Exponential convergence for analytic solutions}

If the solution $(E,H)$ is analytic in a complex neighbourhood of
some mesh elements, for these elements one can replace the approximation bound
\eqref{eq:FinalBestApproxNEW} with the exponential one \eqref{eq:ExpApprox}
in the last step of the proof of Theorem~\ref{th:Convergence}. 
This suggests the design of suitable $hp$-mesh refinements which can provide exponentially convergent Trefftz-DG discretisations.

In the rest of this section, we discuss sufficient conditions on the problem data such that 
analyticity of the solution is guaranteed in a complex neighbourhood of
all the elements in the mesh. 

We assume that the coefficients $\epsil$ and $\mu$ are
constant throughout $Q$, and, for simplicity, that the Dirichlet
boundary conditions are homogeneous, i.e.\ $E_L=E_R=0$.

We first note that the solution $(E,H)$ of the initial boundary value
problem \eqref{eq:IVP} 
is the restriction of the solution $(\widetilde
E,\widetilde H)$ of the similar problem posed on
$\IR\times\IR^+$ with initial conditions $\widetilde E(\cdot,0)=\widetilde E_0$ and
$\widetilde H(\cdot,0)=\widetilde H_0$,
where $\widetilde E_0$ is the $2(x_R-x_L)$-periodic extensions of
$E_0$ odd around the point $x_L$, and $\widetilde H_0$ is the $2(x_R-x_L)$-periodic extensions of
$H_0$ even around the same point.

We now assume that 
\begin{equation*}
\text{$\widetilde E_0$ and $\widetilde H_0$ are analytic
in the complex strip }
\calS_r:=\{z\in\IC,\; |\im z|<r \}
\end{equation*}
for some $r>0$.
As in \eqref{eq:EHuw} and \eqref{eq:leftright}, we decompose the
(extended) initial conditions in the components $\widetilde u_0:=\epsil^{1/2}\widetilde E_0+\mu^{1/2}\widetilde
H_0$ and $\widetilde w_0:=\epsil^{1/2}\widetilde E_0-\mu^{1/2}\widetilde
H_0$, which are also analytic in~$\calS_r$.
(What we actually need is only that $\widetilde u_0$ and $\widetilde w_0$
are analytic in a sufficiently large complex neighbourhood of the
\emph{finite} segments $\Omega_Q^-$ and $\Omega_Q^+$, respectively.)


For every mesh element $K$ as above, we
fix $h_K:=\mathrm{length}(\Omega_K^\pm)$.
The complex ellipses with foci at the extrema of $\Omega_K^\pm$ and sum of the
semiaxes equal to 
$r+\sqrt{r^2+h_K^2/4}$
are contained in the strip $\calS_r$.
So the exponential approximation bound \eqref{eq:ExpApprox} holds with $D$ chosen as $K$ and exponential rate
\begin{equation}\label{eq:Rho}
\rho_K:=2r/h_K+\sqrt{1+4 r^2/h_K^2}>1.
\end{equation}


Proceeding as in the proof of Theorem \ref{th:Convergence}, we see
that the solution of the Trefftz-DG discretisation converges with
exponential rates: 
\begin{align}
\nonumber
\Tnorm{(E,H)-(\Ehp,\Hhp)}_{DG}&\le
3\,\sqrt2\sum_{K\in\calT_h}
\left(h_K^x+2\zeta_K h_K^t\right)^{1/2}
C_{K,\mathrm{Bern}} \,\rho_K^{-p_K},\\
\Big(\N{\mu^{-1/2}(E-\Ehp)}_{L^2(Q)}^2
&+\N{\epsil^{-1/2}(H-\Hhp)}_{L^2(Q)}^2\Big)^{1/2}
\label{eq:ExpConvergence}\\
&\le 6 \,\Cstab\sum_{K\in\calT_h}
\left(h_K^x+2\zeta_K h_K^t\right)^{1/2}
C_{K,\mathrm{Bern}}\, \rho_K^{-p_K},
\nonumber
\end{align}
where $C_{K,\mathrm{Bern}},\Cstab,\rho_K,\zeta_K$ were defined
in \eqref{eq:Bernstein}, \eqref{eq:DualStability}, \eqref{eq:Rho}, and
\eqref{eq:zeta} respectively.
In particular, since we have taken constant parameters, $\Cstab$
satisfies \eqref{eq:CdualAppendix} in appendix \ref{appendix}.

The bounds \eqref{eq:ExpConvergence} ensure that, under the regularity assumptions stipulated in this section, the Trefftz-DG method converges exponentially in the total number of degrees of freedom on any fixed mesh, when the polynomial degrees $p_K$ are increased uniformly.
On the contrary, standard space--time DG methods converge as negative exponentials of the \textit{square root} of the total number of degrees of freedom, as demonstrated e.g.\ in the numerical example in Figure~\ref{fig:ConvpVersion}.

\begin{rem}\label{rem:RobinAnalytic}
In the case of the Robin initial boundary value problem
\eqref{eq:RobinIVP}, convergence bounds similar to \eqref{eq:ExpConvergence} can be proved if
the functions
\begin{align*}
\widetilde u_0^\calR(x):=
\begin{cases}
g_L\big((x_L-x)/c\big) &\iin (x_L-cT,x_L],\\
\epsil^{1/2}E_0(x)+\mu^{1/2} H_0(x) &\iin (x_L,x_R),
\end{cases}\\
\widetilde w_0^\calR(x):=
\begin{cases}
\epsil^{1/2}E_0(x)-\mu^{1/2} H_0(x) &\iin (x_L,x_R),\\
g_R\big((x-x_R)/c\big) &\iin [x_R,x_R+cT)
\end{cases}
\end{align*}
can be extended analytically in $\calS_r$ (or in a sufficiently large complex
neighbourhood of their domains of definition).
Note that, not only the data $g_L,g_R,E_0$, and $H_0$ must be analytic,
but their derivatives also need to match appropriately at the points
$(x_L,0)$ and $(x_R,0)$.
\end{rem}

\section{Numerical experiments}
\label{s:numer-exper}

In this section we present numerical results supporting the
theoretical findings of the preceding sections. In particular, we
present convergence orders for the $h$- and
$p$-versions obtained from a series of numerical experiments. These are compared for
the choice of a Trefftz basis and a non-Trefftz basis.
For the former we employ polynomials of the Trefftz space introduced
in \eqref{eq:localpolyn}, for the latter we choose a tensor product of
Legendre polynomials in the spatial
and temporal variables, respectively.


\subsection{Numerical method}
\label{sec:numerical-method}

For all numerical experiments we employ formulation~\eqref{eq:BilinTDG} with the flux stabilisation parameters in~\eqref{eq:Fluxes} chosen as $\alpha=\beta=1/2$ unless stated otherwise. 
In matrix form the resulting numerical scheme reads
\begin{equation}
\label{eq:NumericalSystem}
\mathbf{A} \mathbf{f}^{n+1} = \mathbf{R} \mathbf{f}^{n},
\end{equation}
where $\mathbf{f}^{n}$ is the vector of numerical degrees of freedom
at time step $n$, and $\mathbf{A}$ and $\mathbf{R}$ are assembled from terms of the
bilinear form \eqref{eq:BilinTDG} acting on $\mathbf{f}$ at step $n+1$ and $n$, respectively.

Following \eqref{eq:localpolyn} the Trefftz basis consists of transport polynomials.
As non-Trefftz basis functions we take
\begin{equation*} 
 v^{j_x,j_t}_{E,H} \left(x,t\right) = L_{j_x} (x) L_{j_t} (t),
\end{equation*}
where $L_j(x)$ denotes a Legendre polynomial of degree $j$. 
The polynomial degrees in space $j_x$ and time $j_t$ are independent
of one another and chosen such that $j_x+j_t \le j$.
The dimensions of the Trefftz and the full polynomial space are given
in section~\ref{ss:PolyTrefftz}.
The choice of complete polynomials is motivated by the observation
that for transport polynomials the degrees in the spatial and temporal
variable add up to $j$ as well, and by the fact that for DG schemes
complete polynomial spaces deliver similar accuracy as tensor product ones with less degrees of freedom,
as demonstrated in \cite{CGH14}.

\subsection{Test problem}
\label{sec:test-problem}

In the following, we consider meshes composed by space--time squares of uniform sizes, partitioning the (1+1)-dimensional domain $Q = [0,60]^2$.
We enforce Dirichlet boundary conditions representing a perfect
electrical conductor (PEC), i.e.\ $E_L = E_R = 0$ in~\eqref{eq:IVP}.

As initial conditions we choose
\begin{equation*}
 \begin{pmatrix}E_0 \\ H_0  \end{pmatrix} \left( x \right)
 = \begin{pmatrix} 1\\1 \end{pmatrix} \exp \left( - \left( x - 10
   \right)^2 / 10 \right),
\end{equation*}
corresponding to a wave packet of Gaussian shape propagating in
positive direction in free space.
The wave is reflected once by the domain boundary.
We normalise to one the permittivity $\epsil$, the permeability $\mu$, and thus the speed of light $c$.

\subsection{Error convergence}
\label{sec:error-convergence}

We consider the relative error computed in the $L^2(Q)$ norm on the whole
space--time domain
\begin{equation}
  \label{eq:L2error}
  \epsilon_Q = \bigg(
    \iint_Q\big( (E-\Ehp)^2 +(H-\Hhp)^2 \big) \di x \di t 
    \bigg/ \iint_Q \big(E^2+H^2\big) \di x \di t    \bigg)^{1/2}.
\end{equation}
We first consider convergence of the $p$-version for both
Trefftz and non-Trefftz basis functions.  The mesh step sizes are $h_x = h_t = 1$.
In Figure~\ref{fig:ConvpVersion} (left) we display the global relative error $\epsilon_Q$
in dependence of the polynomial degree $p$. In both cases spectral
convergence is observed.
However, the Trefftz space of order $p$ has a smaller dimension and from bound \eqref{eq:ExpConvergence} 
we expect the error to decrease exponentially in the number of degrees of freedom in the Trefftz case, but exponentially only in the square root of the number of degrees of freedom in the non-Trefftz case. 
This is observed in the middle and right panels of Figure~\ref{fig:ConvpVersion}.

\begin{figure}[htb]
\centering
\includegraphics[width=\textwidth, clip=true, trim=8mm 0 10mm 0]{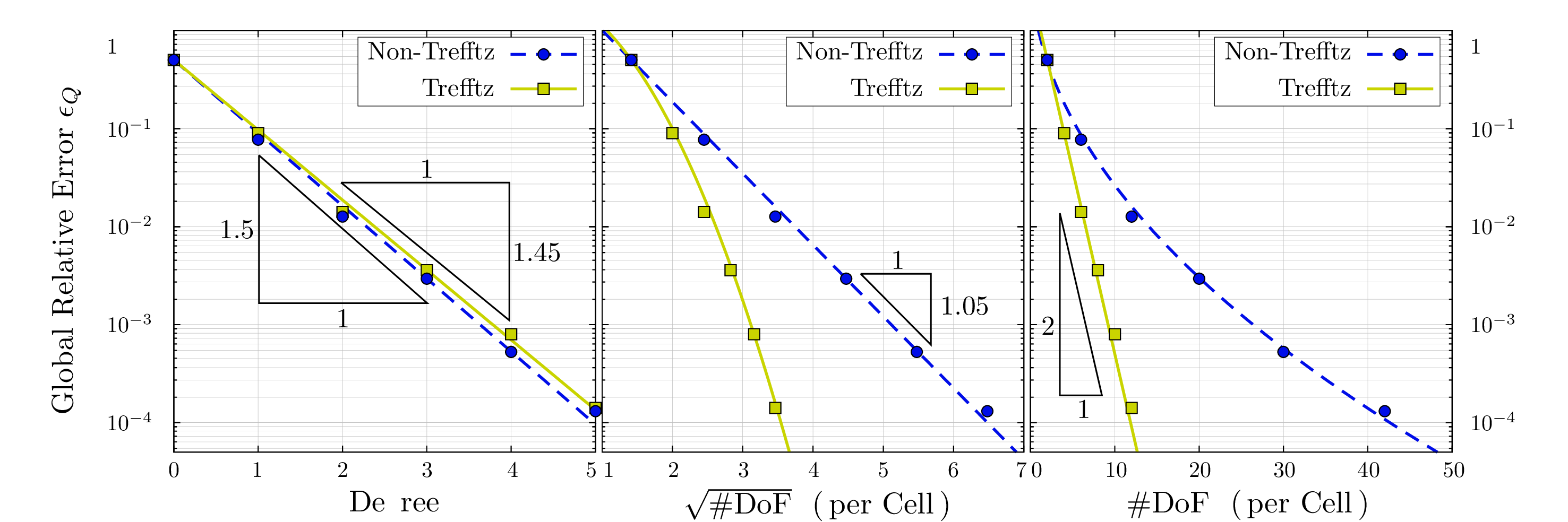}
\caption[fig:ConvpVersion]{Convergence of the $p$-version. Comparison of the
  global relative $L^2(Q)$ error~\eqref{eq:L2error} for Trefftz and non-Trefftz basis
  functions for the propagation of a 
  smooth wave packet in (1+1)-dimensional
  space against the polynomial degree (left), 
  the square root of the number of degrees of freedom per element (middle), and the number of degrees of freedom per element (right), respectively.
  \label{fig:ConvpVersion}
  }
\end{figure}

Figure~\ref{fig:ConvhVersion} shows convergence of the
$h$-version for degree zero through three. Solid lines correspond to results obtained with the Trefftz basis whereas the dashed lines were obtained using the non-Trefftz basis. 
Uniform mesh step sizes are applied by reducing $h_x$ and $h_t$ simultaneously.
The Trefftz method exhibits optimal algebraic convergence rates $h^{p+1}$.
However, in the non-Trefftz case, the results seem to suggest an odd-even
pattern of the convergence rates, with convergence being suboptimal for odd degrees (by one order). 
Numerical odd-even effects in the convergence rates of DG methods
have also been reported, e.g.\ in~\cite[section~6.5]{Hesthaven2008},
although it has been shown in~\cite{RiviereGuzman} that in some
situations this might be a mesh effect, the convergence being always
suboptimal on particular meshes.



\begin{figure}[htb]
\centering
\includegraphics[width=0.9\textwidth]{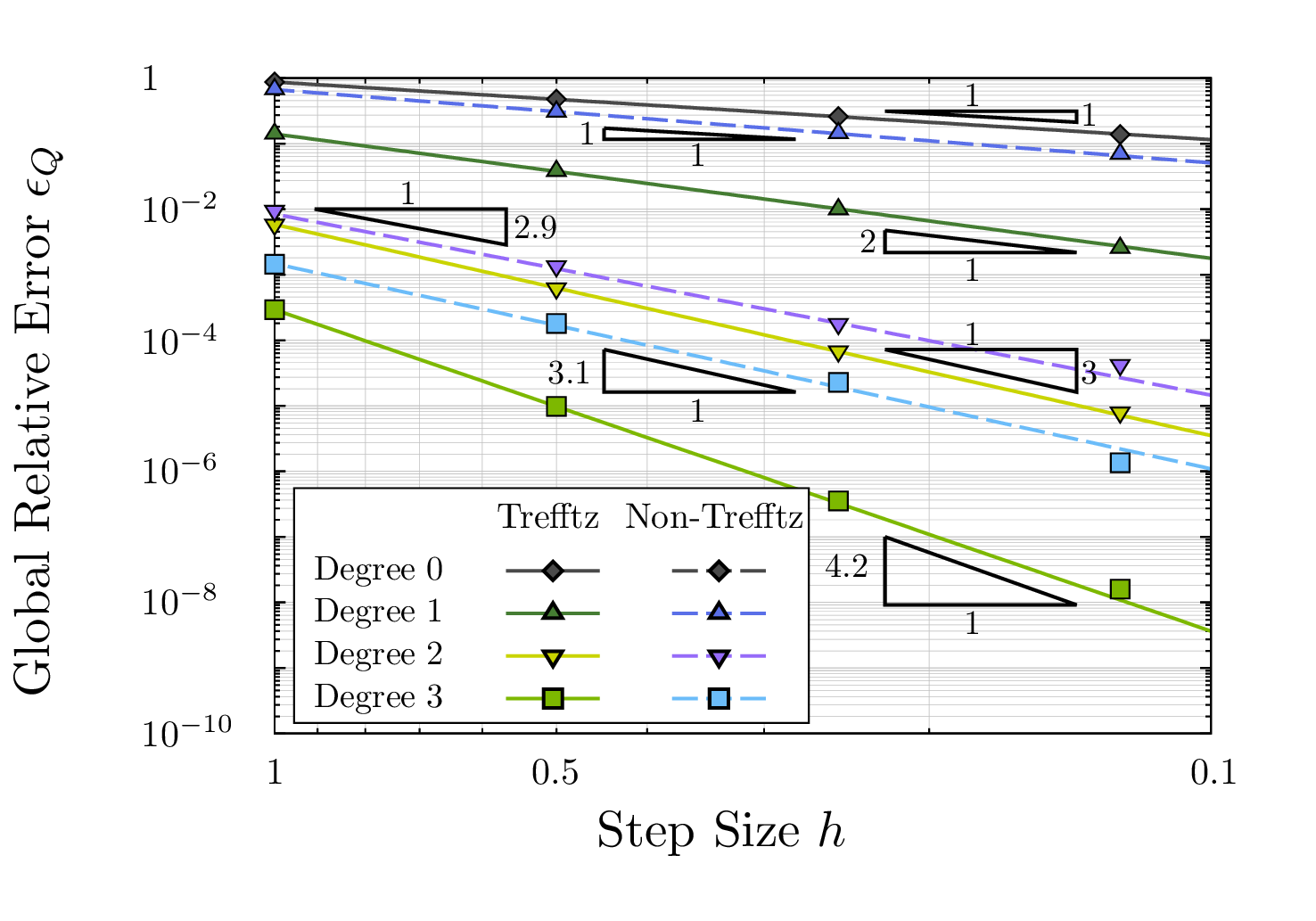}
\caption[fig:ConvhVersion]{Convergence of the $h$-version. Comparison of the
  global relative $L^2(Q)$ error~\eqref{eq:L2error} for Trefftz and non-Trefftz basis
  functions for the propagation of a 
  smooth wave packet in (1+1)-dimensional
  space vs.\ the mesh step size. The spatial and temporal step size
  $h_x$ and $h_t$ are decreased simultaneously. Optimal convergence is
  obtained with the Trefftz basis. 
  In the non-Trefftz case, suboptimal orders of convergence
  are observed for odd polynomial degrees.
  }\label{fig:ConvhVersion} 
\end{figure}

\subsection{Stability}
\label{sec:stability}

As the space--time Trefftz-DG method is implicit in time, a linear
system of equations has to be solved for advancing the solution in every time step. In this
regard, we investigated the conditioning of the Trefftz and
non-Trefftz system matrices. 
Figure~\ref{fig:conditioning} shows that the increase of the conditioning with the polynomial degree
in the Trefftz case is very mild, compared to the non-Trefftz case.
\begin{figure}[ht]
\centering
\includegraphics[width=0.6\textwidth]{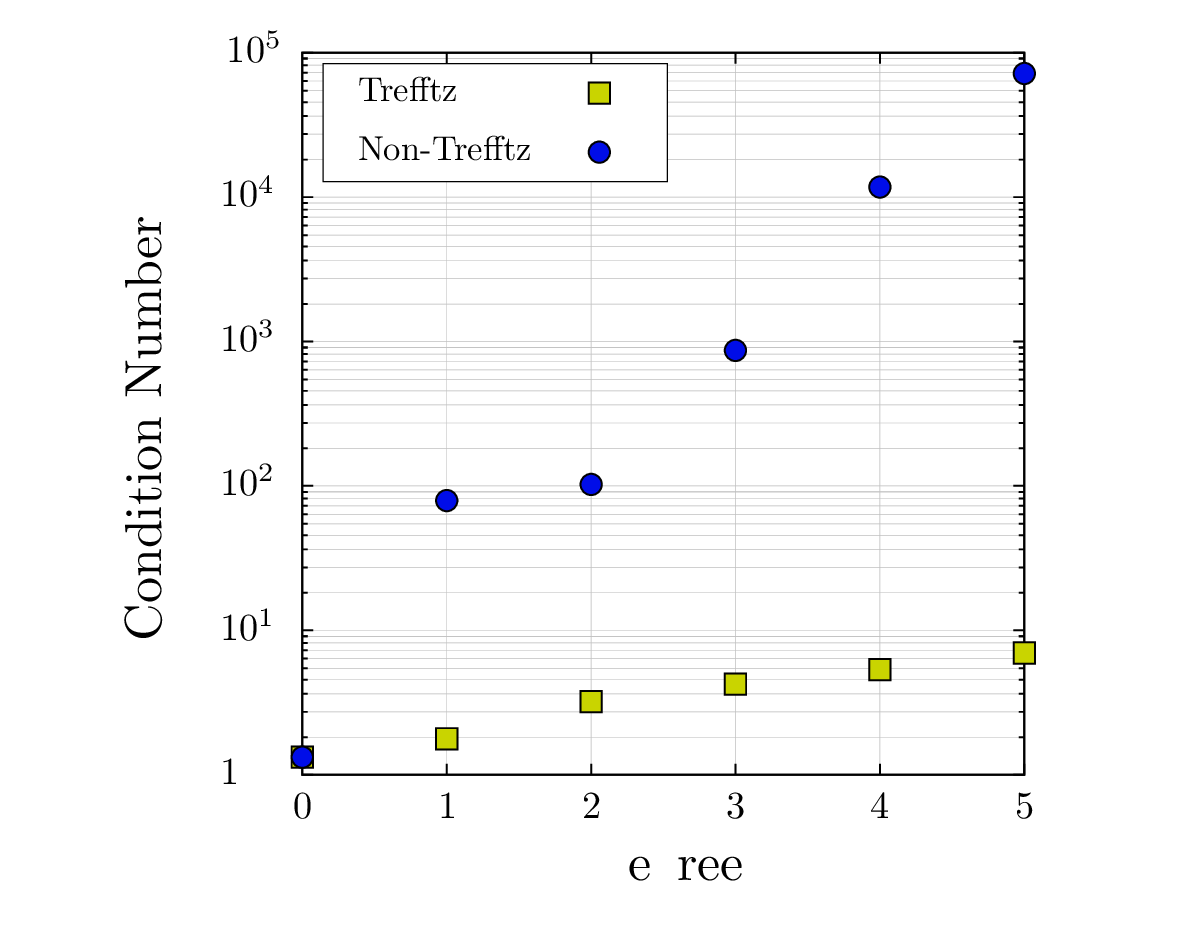}
\caption[fig:conditioning]{Comparison of the condition number of
  the update matrix in the Trefftz  and non-Trefftz case in dependence
  of the polynomial degree.}\label{fig:conditioning}
\end{figure}

The update matrix is obtained from~\eqref{eq:NumericalSystem} as $\mathbf{U}=\mathbf{A}^{-1} \,
\mathbf{R}$. Note that this matrix is usually not explicitly assembled but we
solve for the system~\eqref{eq:NumericalSystem}.
In Figure~\ref{fig:update} we show eigenvalues of the update matrices 
with Trefftz basis for degrees $p=0,\ldots,5$.
As expected from the stability analysis of section~\ref{ss:Energy}, all eigenvalues are on or within the unit circle. 
Figure~\ref{fig:update} numerically confirms stability, and it shows
the method to be dissipative. 


\begin{figure}[ht]
\centering
\includegraphics[width=0.99\textwidth]{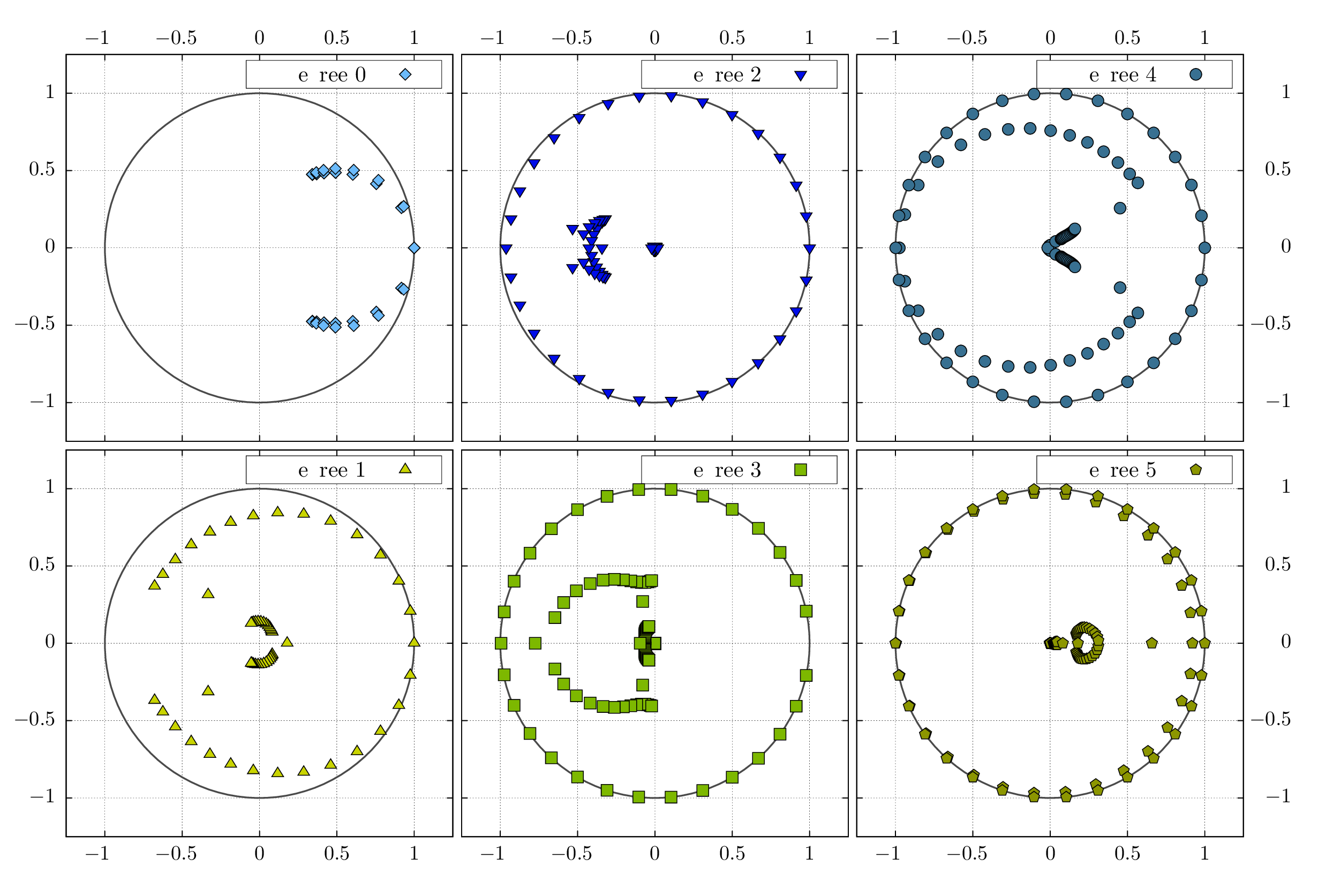}
\caption[fig:update]{Eigenvalues of update matrices with Trefftz basis and flux stabilisation parameters  $\alpha=\beta=0.5$. 
  The grey solid line is the unit circle in the complex plane.
}\label{fig:update}
\end{figure}

\subsection{Flux stabilisation parameters}
\label{sec:flux-stab-param}

Let us shortly comment on the choice of the flux stabilisation parameters $\alpha, \beta$.
The choice $\alpha=\beta=1/2$ corresponds to a full upwind
flux~\cite{MoRi05} and $\alpha=\beta=0$ (see \cite{KSTW2014} and also \cite{Lilienthal:2014fs})
corresponds to a centred flux. 
In order to determine the influence onto the numerical error, we varied
$\alpha, \beta$ in steps of 0.1 in the range $[0,1]$ and repeated the
above example maintaining the same regular mesh with $h_x=h_t=1$
  and polynomial degree $p=2$.
Figure~\ref{fig:alpha_beta} depicts the respective global relative error $\epsilon_Q$.
The minimum error was obtained for $\alpha=1/2$ and
$\beta=0$. 
However, the overall variation of the error is within a factor of two. 
Similar results are obtained for different settings of the mesh step size and the polynomial degree. 
 \begin{figure}[!!ht]
 \centering
 \includegraphics[width=\textwidth, clip=true, trim=0 12mm 0 8mm]{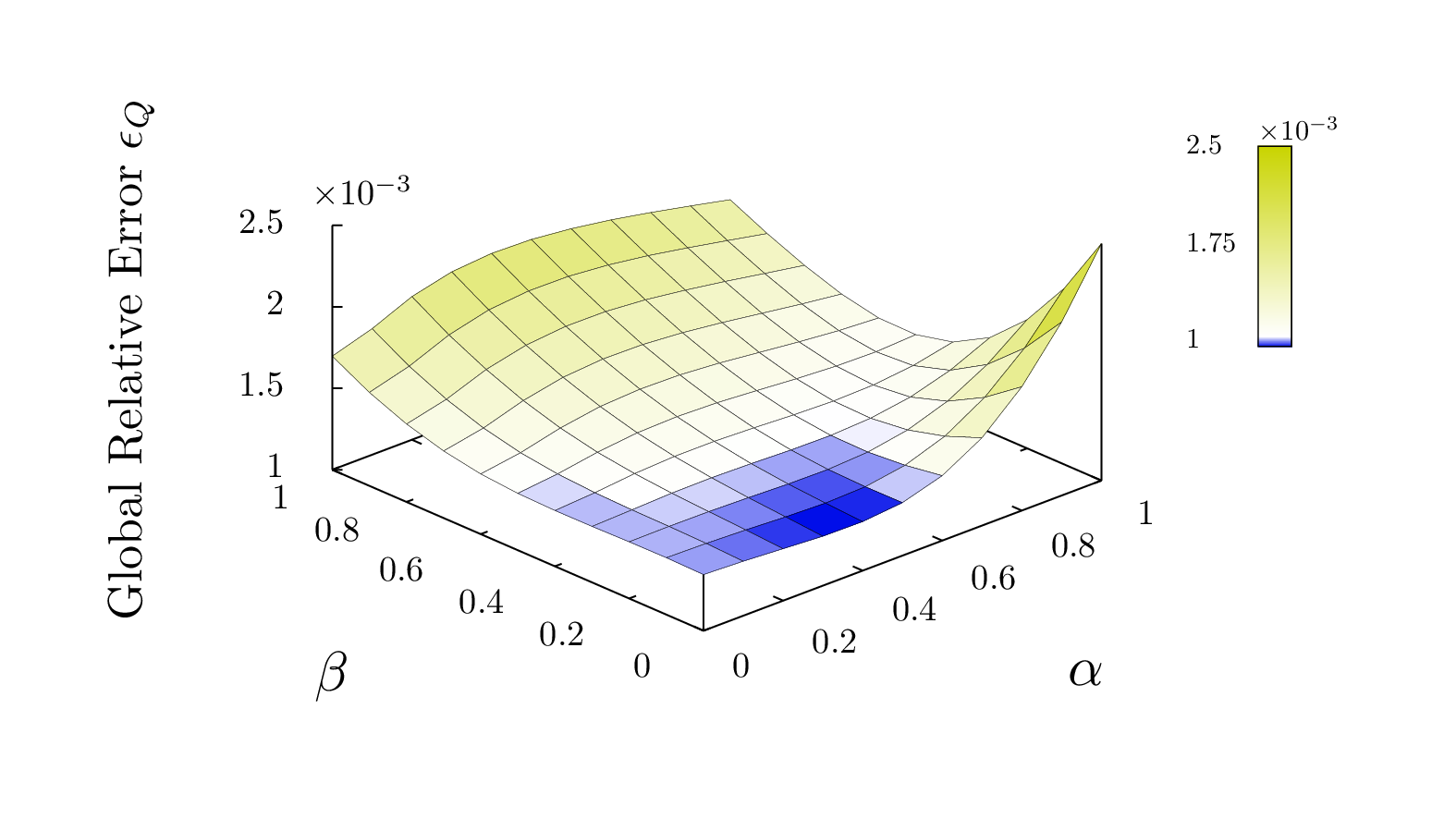}
 \caption[fig:alpha_beta]{Comparison of the global relative $L^2(Q)$
   error~\eqref{eq:L2error} for the setup described in section \ref{sec:test-problem}, under variation of the stabilisation parameters $\alpha$ and $\beta$ for a Trefftz basis
   of degree two (see section \ref{sec:flux-stab-param}).}\label{fig:alpha_beta}
 \end{figure}

\section{Conclusions and extensions}
\label{s:conclusions}

In this paper we have analysed a space--time Trefftz discontinuous Galerkin method for linear wave propagation problems.
The scheme corresponds to that proposed in \cite{KSTW2014}, with the addition of jump penalisation terms.
In one space dimension we proved that the formulation is well-posed and dissipative, and, for polynomial Trefftz trial spaces, we derived a priori $hp$-convergence bounds for its error in DG and $L^2$ norm.
Numerical examples show the viability of the scheme and confirm the orders of converge.

Several extensions of the analysis developed here can be envisaged.
In higher space dimensions, the well-posedness and the
abstract error analysis in DG norm are
straightforward. 
In order to prove orders of convergence for the scheme in higher dimensions, new best approximation bounds for Trefftz spaces must be developed, as those derived in section~\ref{s:BestApprox} rely on the exact representation of the PDE solution in terms of left- and right-propagating waves, which is a one-dimensional result.

Other topics deserving further investigation are
the generalisation of the analysis to unstructured meshes (see
\cite{MoRi05,CostanzoHuang05}) and
the derivation of approximation estimates for non-polynomial Trefftz bases.
In addition, it would be interesting to analyse the non-penalised case with $\alpha=\beta=0$, as well as other possible less dissipative (or non dissipative) variants.

\section*{Acknowledgements}
The authors are grateful to Lehel Banjai for fruitful discussions on the analysis of space--time Trefftz methods.
Ilaria Perugia acknowledges support of the Italian Ministry of Education, University and Research (MIUR) through the project PRIN-2012HBLYE4.
The work of Fritz Kretzschmar is supported by the `Excellence Initiative' of the German Federal and State Governments and the Graduate School of Computational Engineering at Technische Universit\"at Darmstadt.

\appendix
\section{Stability bound in the case of constant coefficients}\label{appendix}

In the special case of constant material parameters $\epsil$ and $\mu$, the stability
bound \eqref{eq:DualStability} can be derived differently from
Lemma~\ref{lem:DualStability}, namely using an exact representation
of the solution of~\eqref{eq:dualIVP}.
This results in a simpler 
constant $\Cstab$ with linear, as opposed to exponential, dependence on the total time $T$;
no additional assumptions on the flux parameters
$\alpha$ and $\beta$ are required.
\begin{lemma}\label{lem:DualStabilityAppendix}
Assume that $\epsil$ and $\mu$ are constant in $\Omega$ (and thus in $Q$).
The solution $(\vE,\vH)$ of the initial auxiliary problem \eqref{eq:dualIVP} satisfies the stability bound 
\eqref{eq:DualStability} with 
\begin{equation}\label{eq:CdualAppendix}
\Cstab^2\le  4Tc\big(cN_{\mathrm{hor}} + \eta\gamma
N_{\mathrm{ver}}\big),
\end{equation}
where 
\begin{align*}
N_{\mathrm{hor}}:=&\;\#\big\{t, \text{ such that } (x,t)\in\Fhor\cup\FT 
\text{ for some }x_L<x<x_R\big\},\\
N_{\mathrm{ver}}:=&\;\#\big\{x, \text{ such that } (x,t)\in\Fver\cup\FL\cup\FR
\text{ for some }0<t<T\big\},\\
\eta:=&\lceil cT/(x_R-x_L)\rceil,\\
\gamma:=&\max\Big\{\N{(\beta\epsil)^{-1}+(\alpha\mu)^{-1}}_{L^\infty(\Fver)};\,
\N{(\alpha\mu)^{-1}}_{L^\infty(\FL\cup\FR)}\Big\}.
\end{align*}
\end{lemma}
\begin{proof}
We assume that $\phi$ and $\psi$ are continuous in $Q$; the general case will follow by a density argument.

First, we extend the initial problem to the entire space $\IR$.
Define $\tvE,\tvH,\tphi,\tpsi$ in $\IR\times\IR^+$ as the $2(x_R-x_L)$-periodic functions in $x$ that satisfy $\tvE|_Q=\vE,\tvH|_Q=\vH,\tphi|_Q=\phi,\tpsi|_Q=\psi$ and such that $\tvE$ and $\tpsi$ are odd around $x_L$ (and consequently also around $x_R$), and  $\tvH$ and $\tphi$ are even around the same points, i.e.\
\begin{align*}
\tvE(x_L+x,t)&=-\tvE(x_L-x,t),\qquad
\tvH(x_L+x,t)=\tvH(x_L-x,t),\\
\tphi(x_L+x,t)&=\tphi(x_L-x,t),\qquad
\tpsi(x_L+x,t)=-\tpsi(x_L-x,t),\qquad \forall (x,t)\in\IR\times\IR^+.
\end{align*}
(Note that the absolute values are $(x_R-x_L)$-periodic in $x$.)
Since time derivatives preserve parities and space derivatives swap them, the extended functions $\tvE$ and $\tvH$ are continuous and satisfy the extended initial problem
\begin{align*}
&\der{\tvE}x+\der{(\mu \tvH)}t = \tphi &&\iin \IR\times\IR^+,\\
&\der{\tvH}x+\der{(\epsil \tvE)}t = \tpsi &&\iin \IR\times\IR^+,\\
&\tvE(\cdot,0)=0, \quad
\tvH(\cdot,0)=0 &&\oon \IR.
\end{align*}

Second, we split the right- and the left-propagating components.
Define 
$$u:=\epsil^{1/2} \tvE+\mu^{1/2} \tvH, \quad w:=\epsil^{1/2} \tvE-\mu^{1/2} \tvH,
\qquad\text{so that}\qquad
\tvE=\frac{u+w}{2\epsil^{1/2}}, \quad \tvH=\frac{u-w}{2\mu^{1/2}}.
$$
They satisfy the inhomogeneous transport equations in $\IR\times\IR^+$
$$
\der u x+\der{(c^{-1}u)}t=\epsil^{1/2} \tphi+\mu^{1/2}\tpsi=:f, \qquad
\der w x-\der{(c^{-1}w)}t=\epsil^{1/2} \tphi-\mu^{1/2}\tpsi=:g,
$$
recalling that $(\epsil\mu)^{1/2}=c^{-1}$, so they can be written explicitly with the following representation formula (e.g.\ \cite[section~2.1.2, equation~(5)]{EVA02}, recall that from the assumptions made in the proof, $f$ and $g$ are piecewise continuous)
$$u(x,t)=\int_0^t cf\big(x+c(s-t),s\big)\di s,\qquad
w(x,t)=-\int_0^t cg\big(x-c(s-t),s\big)\di s .$$

We first bound the $L^2$~norm of $u$ and $w$ on horizontal and vertical segments with the data $f,g$; from the triangle inequality $(\tvE,\tvH)$ will be bounded by $\tphi$ and $\tpsi$, and the bound for $\vE$ and $\vH$ will follow.
For all $0\le t\le T$
\begin{align*}
\N{u(\cdot,t)}^2_{L^2(\Omega)}
&=\int_{x_L}^{x_R}  \Big(\int_0^t cf\big(x+c(s-t),s\big)\di s  \Big)^2 \di x\\
&\le tc^2\int_{x_L}^{x_R}  \int_0^t \big|f\big(x+c(s-t),s\big)\big|^2\di s\di x\\
&= tc^2\int_0^t \int_{x_L+c(s-t)}^{x_R+c(s-t)}  \big|f(y,s)\big|^2\di y\di s\\
&\le
2tc^2\int_0^t \int_{x_L+c(s-t)}^{x_R+c(s-t)} \big(\epsil|\tphi(y,s)|^2+\mu|\tpsi(y,s)|^2\big)\di y\di s
\\
&= 2tc^2 \bigg(\N{\epsil^{1/2}\phi}^2_{L^2(\Omega\times(0,t))}+
\N{\mu^{1/2}\psi}^2_{L^2(\Omega\times(0,t))}\bigg)
\end{align*}
(the last equality follows from the symmetries of $\tphi$ and $\tpsi$ which ensure the equality of their $L^2$~norms on the rectangle $(x_L,x_R)\times(0,t)$ and on the parallelogram with vertices $(x_L-ct,0),(x_R-ct,0),(x_R,t), (x_L,t)$).
Similarly, for all $x\in\Omega$
\begin{align*}
\N{u(x,\cdot)}^2_{L^2(0,T)}
&=\int_0^T  \Big(\int_0^t cf\big(x+c(s-t),s\big)\di s  \Big)^2 \di t\\
&\le c^2\int_0^T  t\int_0^t \big|f\big(x+c(s-t),s\big)\big|^2\di s\di t\\
&= c^2\int_0^T \int_s^T t\big|f\big(x+c(s-t),s\big)\big|^2\di t\di s\\
&= c^2\int_0^T \int_{x-c(T-s)}^{x} \Big(\frac{x-y}c+s\Big) \big|f(y,s)\big|^2\frac1c\di y\di s\\
&\le 2Tc\int_0^T \int_{x-c(T-s)}^{x}\big(\epsil|\tphi(y,s)|^2+\mu|\tpsi(y,s)|^2\big)\di y\di s\\
&\le 2Tc \bigg(\N{\epsil^{1/2}\tphi}^2_{L^2((x-cT,x)\times(0,T))}
+\N{\mu^{1/2}\tpsi}^2_{L^2((x-cT,x)\times(0,T))}\bigg)\\
&\le 2Tc \bigg\lceil\frac{cT}{x_R-x_L}\bigg\rceil \bigg(\N{\epsil^{1/2}\phi}^2_{L^2(Q)}+\N{\mu^{1/2}\psi}^2_{L^2(Q)}\bigg),
\end{align*}
where $\lceil cT/(x_R-x_L)\rceil$ is the number of times $Q$ must be replicated in the $x$ direction to cover the left domain of dependence of $\{x\}\times(0,T)$.

Analogous bounds can be proved for the left-propagating term $w$.

Going back to the solution of the auxiliary problem, we obtain
(using $\epsil\vE^2\le(u^2+w^2)/2$, $\mu\vH^2\le(u^2+w^2)/2$ and summing over all vertical and horizontal segments)
\begin{align*}
&\N{\epsil^{1/2}\vE}^{2}_{L^2(\Fhor\cup{{\FT}})}
+\N{\mu^{1/2}\vH}^{2}_{L^2(\Fhor\cup{{\FT}})}\\
&\hspace{10mm}+\N{{\beta^{-1/2}}\vE}^{2}_{L^2(\Fver)}
+\N{{\alpha^{-1/2}}\vH}^{2}_{L^2(\Fver\cup\FL\cup\FR)}
\\
&\le \N{u}^{2}_{L^2(\Fhor\cup\FT)}
+\N{w}^{2}_{L^2(\Fhor\cup\FT)}
+\N{\gamma^{1/2} u}^{2}_{L^2(\Fver\cup\FL\cup\FR)}
+\N{\gamma^{1/2} w}^{2}_{L^2(\Fver\cup\FL\cup\FR)}
\\
&\le4Tc\big(cN_{\mathrm{hor}} + \lceil cT/(x_R-x_L)\rceil
\gamma N_{\mathrm{ver}}\big)
\bigg(\N{\epsil^{1/2}\phi}^2_{L^2(Q)}+\N{\mu^{1/2}\psi}^2_{L^2(Q)}\bigg).
\end{align*}
\end{proof}

\end{document}